\RequirePackage{xcolor}
\documentclass[journal,twoside,web]{ieeecolor}
\usepackage{generic}

\usepackage{etoolbox}
\newtoggle{arXiv-version}

\let\labelindent\relax

\usepackage[ruled]{algorithm2e}
\usepackage[noend]{algpseudocode}

\usepackage[inline]{enumitem}
\usepackage{url}
\usepackage{tikz}
\usepackage{amsmath}
\usepackage{amssymb}
\usepackage{amsfonts}
\usepackage{nccmath}
\usepackage{xfrac}

\usepackage{amsthm}
\theoremstyle{plain}
\newtheorem{theorem}{Theorem}[section]
\newtheorem{lemma}[theorem]{Lemma}
\newtheorem{proposition}[theorem]{Proposition}
\newtheorem{corollary}[theorem]{Corollary}

\theoremstyle{definition}

\newtheorem{remark}{Remark}
\newtheorem{example}{Example}
\makeatletter
\renewenvironment{proof}[1][\proofname]{%
  \par\pushQED{\qed}\normalfont%
  \topsep6\p@\@plus6\p@\relax
  \trivlist\item[\hskip\labelsep\bfseries\itshape#1\@addpunct{.}]%
  \ignorespaces
}{%
  \popQED\endtrivlist\@endpefalse
}
\makeatother

\usepackage{mathtools}
\usepackage[margin=2cm]{geometry}
\usepackage{graphicx}
\graphicspath{ {./images/} }
\usepackage{subcaption}
\usepackage{float}
\usepackage[noadjust]{cite}
\usepackage{balance}
\usepackage[noabbrev]{cleveref}
\crefformat{equation}{(#1)}
\usepackage{bbm}
\usepackage{bm}
\usepackage{comment}
\usepackage{arcs}

\usepackage{resizegather}

\DeclareMathOperator{\lambdamax}
{\text{$\overline{\lambda}$}}
\DeclareMathOperator{\lambdamin}{\text{$\underline{\lambda}$}}
\DeclareMathOperator*{\argmin}{\arg\min}

\usepackage{acronym}
\newacro{lqr}[LQR]{Linear Quadratic Regulators}
\newacro{slqr}[SLQR]{Structured \ac{lqr}}
\newacro{olqr}[OLQR]{Output-feedback \ac{lqr}}
\newacro{lqg}[LQG]{Linear Quadratic Gaussian}
\newacro{dare}[DARE]{Discrete-time Algebraic Riccati Equation}
\newacro{ouralgo}[QRNPO]{Quasi Riemannian-Newton Policy Optimization}
\newacro{pg}[PGD]{Projected Gradient Descent}
\newacro{npg}[NPGD]{Natural Projected Gradient Descent}
\newacro{PO}[PO]{Policy Optimization}
\newacro{ode}[ODE]{Ordinary Differential Equations}

\usepackage{cite}
\usepackage{textcomp}
\def\BibTeX{{\rm B\kern-.05em{\sc i\kern-.025em b}\kern-.08em
		T\kern-.1667em\lower.7ex\hbox{E}\kern-.125emX}}
    
\usepackage{enumerate}

\newcommand\x{{\bm x}}
\newcommand\uu{{\bm u}}

\newcommand\y{{\bm y}}

\newcommand\bR{{\mathbb{R}}}
\DeclareMathOperator{\grad}{grad}
\DeclareMathOperator{\hess}{Hess}
\DeclareMathOperator{\euchess}{\text{$\overline\hess$}}
\DeclareMathOperator{\tproj}{\pi^\top}
\DeclareMathOperator{\nproj}{\pi^\perp}
\DeclareMathOperator{\eucproj}{Proj}
\newcommand\metric{g}
\newcommand\submetric{\protect\scalebox{0.9}{$\widetilde{\metric}$}}
\newcommand\lyap{{\,\mathbb{L}}}
\newcommand\lyapij{dY}
\DeclareMathOperator{\Exp}{exp}
\DeclareMathOperator{\diff}{\mathrm{d}}
\newcommand\nghbd{\mathcal{U}}
\newcommand\wein{\mathbb{W}}
\newcommand\ptrans{\mathcal{P}}
\newcommand\fX{\mathfrak{X}}
\newcommand\stabcer{s}

\newcommand\lyaptrace{{Lyapunov-trace~}}
\newcommand\Amatrices{{\text{M}(n\times n,\mathbb{R})}}
\newcommand\Astable{{\mathcal{M}}}
\newcommand\Bmatrices{{\text{M}(n\times m,\mathbb{R})}}
\newcommand\Cmatrices{{\text{M}(d\times n,\mathbb{R})}}
\newcommand\Lmatrices{{\text{M}(m\times d,\mathbb{R})}}
\newcommand\Kmatrices{{\text{M}(m\times n,\mathbb{R})}}
\newcommand\constraint{{\mathcal{K}}}
\newcommand\stableK{\mathcal{S}}
\newcommand\substableK{\protect\scalebox{0.9}{$\widetilde{\stableK}$}}

\newcommand\Enabla{\overline{\nabla}}
\newcommand\dist{\mathrm{dist}}
\newcommand{\tensor}[3]{\ensuremath{\left\langle #1, #2 \right \rangle}}

\makeatletter
\newcommand{\raisemath}[1]{\mathpalette{\raisem@th{#1}}}
\newcommand{\raisem@th}[3]{\raisebox{#1}{$#2#3$}}
\makeatother
\newcommand{\rnorm}[2]{\ensuremath{\left| #1 \right|_{\raisemath{1.5mm}{\metric_{\!_{#2}}}}}}

\newcommand{\tr}[1]{\ensuremath{\mathrm{tr}\left[ #1 \right]}}
\newcommand{\ie}{{i}.{e}.}
\newcommand{\eg}{{e}.{g}.}

\makeatletter
\newcommand{\algorithmfootnote}[2][\footnotesize]{%
  \let\old@algocf@finish\@algocf@finish% Store algorithm finish macro
  \def\@algocf@finish{\old@algocf@finish% Update finish macro to insert "footnote"
    \leavevmode\rlap{\begin{minipage}{\linewidth}
    #1#2
    \end{minipage}}%
  }%
}
\makeatother
\begin{document}

%%%%%%%%%%% document version: arXiv or IEEE
% use \togglefalse{arXiv-version} or \toggletrue{arXiv-version} after the \begin{document} to generate each version:
% \togglefalse{arXiv-version}
\toggletrue{arXiv-version}
%%%%%%%%%%%%%%%%

\title{Policy Optimization over Submanifolds for Linearly Constrained Feedback Synthesis}
\author{Shahriar Talebi, \IEEEmembership{Member, IEEE} and Mehran Mesbahi, \IEEEmembership{Fellow, IEEE}
\thanks{\iftoggle{arXiv-version}{~}{Manuscript received xx Month 20xx; revised xx Month 20xx; accepted xx Month 20xx. Date of publication xx Month 20xx; date of current version xx Month 20xx.} This work was supported by the AFOSR grant FA9550-20-1-0053 and NSF grant ECCS-2149470. \iftoggle{arXiv-version}{~}{Recommended by Associate Editor xxxxx.} (Corresponding author: Shahriar Talebi.)}
\thanks{The authors are with the William E. Boeing Department of Aeronautics and Astronautics, University of Washington, Seattle, WA 98115 USA. S. Talebi is also with the Department of Mathematics at the University of Washington  (e-mail: shahriar@uw.edu; mesbahi@uw.edu).}
\iftoggle{arXiv-version}{~}{
\thanks{Color versions of one or more figures in this article are available at https://doi.org/xxxxxx/TAC.20xx.xxxxx.}
\thanks{Digital Object Identifier xxxxxxx/TAC.20xx.xxxxxxx}}
}

\maketitle

\begin{abstract}
    In this paper, we study linearly constrained policy optimization over the manifold of Schur stabilizing controllers, equipped with a Riemannian metric that emerges naturally in the context of optimal control problems. We provide extrinsic analysis of a generic constrained smooth cost function, that subsequently facilitates subsuming any such constrained problem into this framework. By studying the second order geometry of this manifold, we provide a Newton-type algorithm that does not rely on the exponential mapping nor a retraction, while ensuring local convergence guarantees. 
    % that exploits this inherent geometry 
    The algorithm hinges instead upon the developed stability certificate and the linear structure of the constraints. We then apply our methodology to two well-known constrained optimal control problems.
    Finally, several numerical examples showcase the performance of the proposed algorithm.
\end{abstract}

 
% %%%%%%%%%%%%%%%%%%%%%%%%%%%%%%%%%%%%%%%%%%%%%%%%%%%%%%%%%%%%%%%%%
\section{Introduction}
\label{sec:intro}
\IEEEPARstart{I}{n} recent years, direct \ac{PO} for different variants of the \ac{lqr} problems have attracted considerable attention in the literature.
In the meantime, \ac{PO} for linearly constrained \ac{lqr}, \eg, state-feedback \ac{slqr} and \ac{olqr}, has been less explored due to the intricate geometry of the respective feasible sets and the non-convexity of the cost function.
While reparameterization of \ac{lqr} to a convex setup is possible for unconstrained cases \cite{mohammadi2021convergence}, in general, trivial constraints directly on the policy become nontrivial and non-convex after such reparameterizations.\footnote{There are exceptions to this statement-like when conditions such as quadratic invariance can be invoked~\cite{rotkowitz2006characterization}.}
Furthermore, the domain of the optimization problems 
for constrained \ac{lqr} (and its variants) are generally not only non-convex \cite{Ackermann1980Parameter} but also disconnected~\cite{Feng2019Exponential}.
As such, there are no guarantees that 
first-order stationary points are necessarily local minima.

Finding the linear output-feedback policy directly for the \ac{olqr} problem was first addressed in \cite{Levine1970determination},
a procedure that involves solving nonlinear matrix equations at each iteration. 
Since then, there has been on-going research efforts to address
this problem adopting distinct perspectives~\cite{anderson1973linear, moerder1985convergence, toivonen1985globally, Makila1987Computational, Toivonen1987Newton, iwasaki1994linear, rautert1997computational, maartensson2009gradient},
{including its computational complexity~\cite{syrmos1997static,blondel1997np,papadimitriou1986intractable}. 
In this direction, first and second order methods have been adopted for solving \ac{slqr} and \ac{olqr} problems (see \eg, \cite{Toivonen1987Newton,Makila1987Computational} and references therein).}
% %
However, these methods often have a number of limitations, including reliance on backtracking line-search techniques at each iteration~{(which may be computationally expensive or infeasible)}, absence of convergence guarantees,
not utilizing the inherent non-Euclidean geometry of the problem, and finally, not offering a setup for handling general linear constraints on the feedback gain.

Recently, state-feedback \ac{lqr} problems have been studied through the lens of first order methods, in both discrete-time \cite{bu2019lqr} and continuous-time \cite{bu2020policy} setups. 
This point of view was initiated when the \ac{lqr} cost was shown to satisfy Polyak-Łojasiewicz (PL) (aka \emph{gradient dominance}) property \cite{fazel2018global}, facilitating a global convergence guarantee of first order methods for this problem--despite its non-convexity. 
Since then, \ac{PO} using first order methods has been investigated for variants of \ac{lqr} problem, such as \ac{olqr} \cite{fatkhullin2020optimizing}, model-free setup \cite{mohammadi2021linear}, and risk-constrained \ac{lqr} \cite{zhao2021global}.
The gradient dominance property, however, is only known to be valid with respect to the global optimum of the unconstrained case, and is not necessarily expected for general constrained \ac{lqr} problems. 
By merely using the first order information of the cost function, \ac{pg} techniques--whenever the projection is possible--can be shown to converge to first order stationary points but with a sublinear convergence rate (\eg, see \cite{bu2019lqr} and \cite{fatkhullin2020optimizing} for \ac{slqr} and \ac{olqr} problems, respectively). 
A sublinear rate is generally unfavorable from a practical point of view, particularly when second order information of the \ac{lqr} cost can be 
utilized.
{Despite computational challenges arising from the non-convexity--and even non-connectedness--of stabilizing feedback gains for general constrained optimal control problems, one may consider developing fast convergent algorithms for efficient identification of feasible local optima.}

% %
Note that {\em some} structure on the policy can be enforced through regularization; however, this approach merely \emph{promotes} structural
constraints and does not address problems considered here, as
constraints are prescribed as a hard requirement for feasibility of the solution (\eg, see \cite{park2020structured} for an approach pertaining to promoting sparsity for \ac{slqr} problem).
Here, we aim to utilize the second order information of the cost function to improve the convergence rate, and at the same time, 
provide a general approach that can address \emph{other} linear constraints
from a unifying perspective; in essence, the proposed approach can
be adopted for problems such as~\ac{slqr} and \ac{olqr},
well-recognizing their inherent computational complexity~\cite{syrmos1997static,blondel1997np}.

By ignoring the geometry of the problem, one may aim to optimize the \emph{linearly} constrained \ac{lqr} cost by directly utilizing first or second order methods. Since the domain is non-convex, this might still be possible{--depending on the problem scenario--}by incorporating an Armijo-type backtracking line-search or requiring that the initial guess be close to the local optima. 
Here, we preclude from incorporating a line-search in order to systematically exploit the geometry inherent in the \ac{lqr} cost. 
{This inherent geometry will be precisely captured by the Riemannian metric defined subsequently in \S\ref{sec:results} (see \Cref{lem:tensor}).}
Incorporating a back-tracking technique is then an immediate extension
of our setup.
Furthermore, by adopting a geometric perspective
in designing direct policy optimization,
we aim to also pave the way for future works on including relevant system theoretic criteria that lead to \emph{nonlinear} constraints on the feedback policy; handling such constraints can significantly
benefit from the intrinsic geometry of the problem as investigated in this work (see \S\ref{sec:conclusion} for an example).
\begin{figure}[t]
     \centering
     \includegraphics[trim = 2 2 0 0, clip, width =0.47\textwidth]{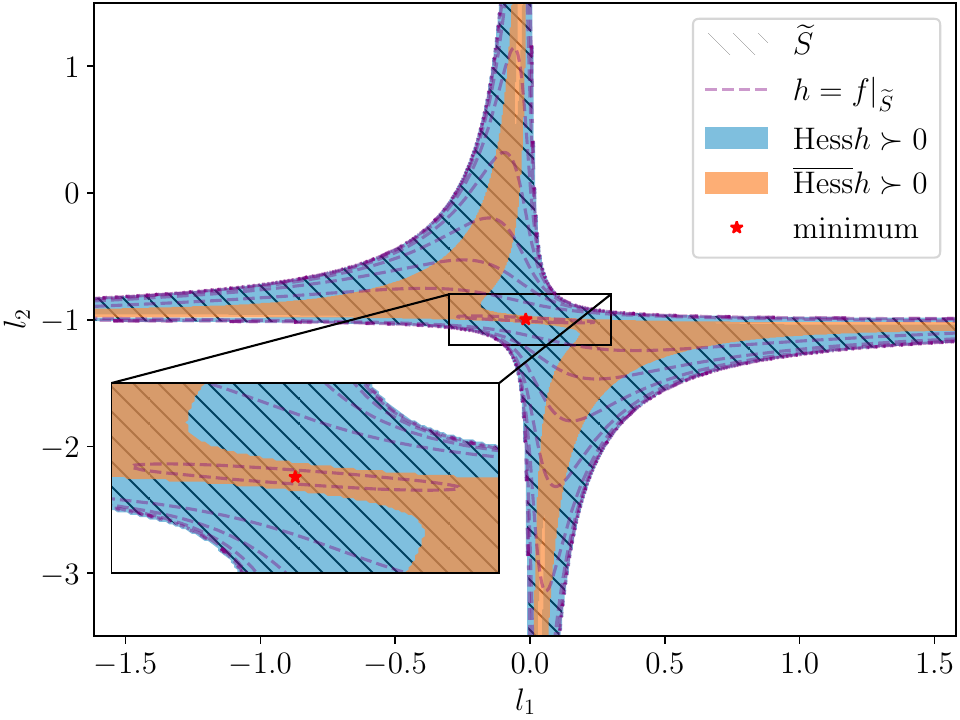}
     \caption{\small
    The submanifold $\substableK$ of  \emph{diagonal} state-feedback stabilizing controllers ($K = \mathrm{diag}(l_1,l_2)$) for a $2 \times 2$ system; superimposed with the level-sets of the constrained \ac{lqr} cost ($h$), and the regions on which its \emph{Hessian} is positive definite with respect to the inherent Riemannian geometry ($\hess{h}$) and the Euclidean geometry ($\euchess{h}$), respectively. Note that $\substableK$ is covered by the region on which $\hess{h} \succ 0$. See \Cref{exp:2} for more details.}
    \label{fig:motivation}
    \vspace{-0.3cm} 
\end{figure}
Generally, the second order behavior of the cost function can be utilized in order to obtain a descent direction--as in the Newton method--as long as the \emph{Hessian} stays positive definite. 
However, this second order information can be obtained in a variety of ways; 
for example, with respect to the usual Euclidean geometry (especially for linearly constrained problems) or more interestingly, with respect to the non-Euclidean geometry inherent to the cost function itself. 
A simple--yet relevant--example of the above statement is depicted in \Cref{fig:motivation} in relation to the set of \emph{diagonally} constrained stabilizing controllers (denoted by $\substableK$)--which turns out to be non-convex even for this simple example consisting of two inputs and two states. 
More specifically, this set is the intersection of the 4-dimensional set of stabilizing controllers with a 2-dimensional plane defining the diagonal constraint. In reference to
\Cref{fig:motivation}, note how the \emph{Euclidean Hessian} of the constrained \ac{lqr} cost (denoted by $\euchess{h}$) is positive definite on a smaller subset of $\substableK$--especially in the vicinity of the ``$\mathrm{minimum}$.'' Therefore, one would expect that the neighborhood of $\mathrm{minimum}$ on which the Newton updates using the Euclidean geometry converge, be relatively small. In the meantime, 
if the second order behavior of the \ac{lqr} cost is considered through the lens of Riemannian geometry, its \emph{Riemannian Hessian} (denoted by $\hess{h}$) captures the behavior of the cost function more effectively, in that it remains positive definite on a larger domain--compared with $\euchess{h}$. Hence, one expects a significant difference in the performance of second order optimization algorithms utilizing these two distinct geometries.
{To further expand on this example, one needs to define 
the (Riemannian) $\hess{h}$ and devise a second order algorithm that iteratively optimize this constrained problem in the vicinity of a local optimum. This machinery is developed in the rest of this paper, after which we revisit the above example (see \Cref{exp:2} in \S\ref{sec:simulation}).}

The literature on optimization over manifolds often relies
on having access to either the exponential mapping \cite{gabay1982minimizing, smith1994optimization}, or a \emph{retraction}\iftoggle{arXiv-version}{\footnote{A better terminology would be \emph{graph projection}; however, we adopt \emph{retraction} to be consistent with the manifold optimization literature.}}\, from its tangent bundle onto the manifold itself \cite{absil2009optimization}.
However, due to the intricate geometry of the manifold of Schur stabilizing controllers, the exponential mapping is computationally 
expensive and a suitable retraction 
is generally not available.
Furthermore, many useful constraints for optimal \ac{lqr} problems (such as \ac{olqr} or \ac{slqr}) inherit a linear structure. Hence, one may consider the Natural Gradient Descent approach of \cite{amari1998natural, amari1998why} in order to utilize this inherent geometry; however, this approach is not directly applicable to more involved submanifolds of stabilizing controllers.
Nonetheless, it is pertinent to ask whether we can still exploit the intrinsic non-Euclidean geometry--induced by, say, the quadratic cost and linear dynamics, in optimal control problems as illustrated in \Cref{fig:motivation}-- by circumventing the absence of a computationally feasible retraction while guaranteeing stability.

In this paper, we consider a general optimization problem over the set of linearly constrained stabilizing feedback gains; this setup can easily be tailored to other classes of constrained control synthesis problems.
We introduce a Newton-type algorithm that utilizes both the inherent Riemannian geometry as well as the linear structure of the constraints, and provide its convergence analysis to the local minima. 
{ Here, in the absence of a computationally feasible (global) retraction from the tangent bundle to the manifold, we obtain the so-called \emph{stability certificate} that--together with the linear structure of the constraints--substitute the
role that a retraction would generally play, 
ensuring the feasibility of the next iterate.} Finally, as the unit stepsize for the proposed iterates may not be possible in general, we guarantee a linear convergence rate--that eventually becomes quadratic as the iterates converge.
Finally, we provide applications of the proposed methodology to the well-known state-feedback \ac{slqr} and \ac{olqr} problems, followed by numerical examples.

Our contributions can thus be summarized as follows:
\begin{enumerate*}[label=(\roman*)]
    \item We study the second order geometry of the manifold of stabilizing controllers induced by a pertinent Riemannian metric and its associated \emph{connection} {(a generalization of directional derivatives \cite{lee2018introduction}), in order to obtain the second order information of a cost through defining the Riemannian Hessian}.
    \item We provide extrinsic analysis for first and second-order behavior of a generic smooth cost function constrained to a \emph{Riemannian submanifold}. This, in turn, allows for a general treatment of constrained optimization problems on the manifold of stabilizing controllers.
    \item We introduce \ac{ouralgo} algorithm (pronounced as \emph{Kern PO}) with convergence guarantees that exploits the inherent Riemannian geometry in the absence of the exponential mapping or a retraction, effectively providing a stability certificate for linearly constrained feedback gains.
    \item We apply our methodology to \ac{slqr} and \ac{olqr} problems by first, computing the second order behavior of the \ac{lqr} cost with respect to the Riemannian connection, and then, explicating the solution to Newton equation for each case using this geometry.
    \item While our approach allows for considering any choice of connection, here, we focus on the associated Riemannian connection--we also make a comparison to the ordinary Euclidean connection.
    \item Finally, we provide several numerical examples to showcase the performance and advantages of the proposed methodology that exploits the intrinsic geometry of constrained feedback stabilization.
\end{enumerate*}
Further applications of the proposed algorithm--\eg~to the problem of controlling diffusion dynamics over a network--has recently appeared in \cite{talebi2022geometric}.

The rest of the paper is organized as follows. 
In \S\ref{sec:probSetup}, we introduce the generic 
(stabilizing) feedback synthesis problem.
We then provide the analysis of this problem through the lens of differential geometry in \S\ref{sec:results}.
In \S\ref{sec:learning-algo}, we present the algorithm and its convergence analysis for optimization on submanifolds of stabilizing controllers.
Applications to \ac{slqr} and \ac{olqr} problems are then presented in \S\ref{sec:application}.
Finally, numerical examples are provided in \S\ref{sec:simulation}, followed by concluding remarks in \S\ref{sec:conclusion}.
The appendix contains proofs of the results.\\
\indent\textbf{Notation:} The space of $m\times n$ matrices over the reals is denoted by $\Kmatrices$ with the trivial smooth structure determined by the atlas consisting of the single chart $(\Kmatrices, \mathrm{vec})$, where $\mathrm{vec}\colon \Kmatrices \mapsto\bR^{mn}$ denotes the operator that returns a vector obtained by (vertically) stacking the columns of a matrix--from left to right. We denote the transpose operator and the spectral norm of a matrix by $(\cdot)^\intercal$ and $\|\cdot\|_2$, respectively. The trace and spectral radius of a square matrix are denoted by $\tr{\cdot}$ and $\rho(\cdot)$. The \textit{Loewner} partial order of symmetric positive (semi-)definite matrices is denoted by $\succ$ ($\succcurlyeq$);
we use the same notation to denote positive (semi-)definiteness of 2-tensor fields.
The maximum and minimum eigenvalues of symmetric matrices will be designated by $\lambdamax$ and $\lambdamin$, respectively.
The set of positive integers less than or equal to $m$ is denoted by $[m]$.
By $\Astable \coloneqq \left\{ A \in \Amatrices \;|\; \rho(A) <1 \right\}$, we denote the set of (Schur) stable matrices, and define the \emph{Lyapunov map}
\(\textstyle \lyap \colon \Astable \times \Amatrices \mapsto\Amatrices,\)
that sends the pair $(A,Z)$ to the unique solution $X$ of 
\begin{equation}\label{eq:lyap-gen}
X = A X A^\intercal + Z,
\end{equation}
which has the representation $\textstyle X = \sum_{i=0}^\infty A^i Z (A^\intercal)^i$;
in this case, if $Z \succeq 0\, (\succ 0)$, then $X \succeq 0 \, (\succ 0)$. Furthermore, when $Z \succeq 0$, then $X \succ 0$ if and only if $(A,Z^{1/2})$ is controllable (see \cite{gajic2008lyapunov} and references therein).
For manifolds we follow the notation and results in \cite{lee2018introduction} and \cite{lee2013smooth} unless stated explicitly.

\section{Problem Statement}
\label{sec:probSetup}
Given a stabilizable pair $(A,B)$ with $A \in \Amatrices$ and $B\in \Bmatrices$, we define
\[ \stableK \coloneqq \{K \in \Kmatrices \;|\; \rho(A+BK) < 1\},\]
as the set of stabilizing feedback gains.
Subsequently, we will introduce a non-Euclidean geometry over $\stableK$ using a metric arising naturally in the context of optimal control problems.
We are often interested in feedback gains $K$ that lie in a relatively \emph{simple} subset $\constraint$ of $\Kmatrices$, such that $\substableK \coloneqq \constraint \cap \stableK$ is an embedded submanifold of $\stableK$ {(see \S\ref{sec:domain-manifold})}. A common example of this would
be a linear subspace of $\Kmatrices$ characterizing a prescribed sparsity pattern for the admissible controller gains (see \S\ref{subsec:slqr}). Another example is the optimal output-feedback synthesis considered in \S\ref{subsec:olqr}.

Herein, we are concerned with the optimization problem,
\begin{align}\label{eq:constained-opt-gen}
   \displaystyle \min_K f(K)  \qquad 
    \text{s.t. } \quad K \in \substableK,
\end{align}
where $f \in {C^{\infty}(\stableK, \mathbb{R}) = C^{\infty}(\stableK)}$ and $\substableK$ is an embedded submanifold of $\stableK$, especially when it is endowed with a linear structure.
This problem is motivated by parameterized feedback synthesis problems where we optimize the \ac{lqr} cost $f$ directly for \emph{the policy} $K \in \substableK$ which takes the form of $f(K) = \frac{1}{2}\tr{P_K \Sigma_K}$ with some $\Sigma_K, P_K \succ 0$ smoothly depending on $K$ (see \S\ref{sec:application} for further details).

Our approach involves using this linear structure with an appropriate Riemannian geometry of $\stableK$ to circumvent the absence of a (computationally feasible) global retraction from $T\stableK$ onto $\stableK$ (or $\substableK$)--due to the intricate geometry of $\stableK$.
{In this paper, we do not explicitly discuss conditions for
the existence of the local (or global) optima for \cref{eq:constained-opt-gen};
as such, we assume that the minimum exists;}
{ see \cite{skogestad2007multivariable,mesbahi2010graph} for a few relevant applications.}

In order to handle a generic embedded submanifold $\substableK$, we study the behavior of the restricted function $h \coloneqq f|_{\footnotesize\substableK}$ from an extrinsic point of view, an approach that
can be generalized to any such submanifold.
We note that in general, the function $f$ is not convex and the constraint submanifold $\substableK$ might be disconnected. {Thus, here we focus on local convergence results that aim to exploit the inherent geometry of the problem in order to achieve fast convergence rates--with a relatively reasonable computational complexity.}
%

% %%%%%%%%%%%%%%%%%%%%%%%%%%%%%%%%%%%%%%%%%%%%%%%%%%%%%%%%%%%%%%%%%
\section{Geometry of the Synthesis Problem}
\label{sec:results}
In order to examine \cref{eq:constained-opt-gen}, we analyze the domain manifold using machinery borrowed from differential geometry. 
Note that embedded submanifolds $\substableK$ endowed with a linear structure can certainly be investigated without using such a machinery. However, neither the corresponding results can be generalized to submanifolds with \emph{nonlinear} structures, nor the geometry induced by the cost function can be exploited for developing the corresponding optimization algorithms.

Before we proceed, it is worth noting that if we were to directly apply the results developed for optimization over the manifolds (such as \cite{absil2009optimization}), it would have been necessary to access a retraction from the tangent bundle $T\substableK$ onto $\substableK$. Unfortunately, due to the intricate geometry of $\stableK$, such a mapping is generally not available. Additionally, we will see that the Riemannian exponential map, with respect to the inherent geometry associated with optimal control problems, involves a system of ordinary differential equations whose coefficients are solutions to different Lyapunov equations. Therefore, even though it is possible to compute the exponential mapping, in general, its computational overhead is hard to justify.
Nonetheless, we show how we can circumvent this issue when the Riemannian tangential projection onto $T \substableK$ is available--an operation that is more streamlined.
\subsection{Analysis of the domain manifold} \label{sec:domain-manifold}
It is known that $\stableK$ is contractible \cite{bu2019topological}, and unbounded when $m \geq 2$ with the topological boundary $\partial \stableK = \{K \in \Kmatrices \;|\; \rho(A+BK) =1\}$ as a subset of $\Kmatrices$. Furthermore, $\stableK$ is open in $\Kmatrices$ {(by continuity of eigenvalues in the entries of the matrix \cite[Theorem 5.2]{serre2010matrices} and passing to the quotient \cite[Theorem 3.73]{lee2010introduction}});
as such $\stableK$ is a submanifold without boundary.\footnote{Cf. \cite{Ohara1992differential} for an analogous study of Hurwitz stabilizing controllers.}

In this paper we focus on $\stableK$ as a manifold on its own. Note that $\stableK$ can be covered by a single smooth chart and the tangent bundle of $\stableK$, denoted by $T \stableK$, is diffeomorphic to $\stableK \times \bR^{mn}$, which in turn is diffeomorphic to $\stableK \times \Kmatrices$ under the map $\mathrm{Id}_{\stableK} \times \mathrm{vec}^{-1}$. 
We refer to this composition of diffeomorphisms as \emph{the usual identification of the tangent bundle} (or $T_K S \cong \Kmatrices$ at any point $K \in \stableK$) if we need to identify any element of $T \stableK$ (or $T_K \stableK$). 
In particular, let us denote the coordinates of this global chart by $(x^{i,j})$ for $\stableK$, its associated global coordinate frame by $(\frac{\partial}{\partial x^{i,j}})$ or simply $(\partial_{i,j})$, and its dual coframe by $(d x^{i,j})$, where $i=1,\dots, m$ and $j=1,\dots, n$. Moreover, the $(k,\ell)$th element of any matrix $A \in \Kmatrices$ is denoted by $[A]_{k,\ell}$ or $[A]^{k,\ell}$ depending on viewing $A$ as a point or a tangent vector, respectively. Then, for example, under the usual identification of tangent bundle, for any fixed $i$ and $j$, we identify $\partial_{i,j}$ as a matrix in $\Kmatrices$ whose elements are $[\partial_{i,j}]^{k,\ell} = 1$ if $k = i$ and $\ell = j$, and otherwise $[\partial_{i,j}]^{k,\ell} = 0$.
We also use the Einstein summation convention as explained in \cite{lee2018introduction} for double indices; for example, we write $x^{i,j}\partial_{i,j}$ to denote $\sum_{i=1}^m \sum_{j=1}^n x^{i,j}\partial_{i,j}$.%
\footnote{Note that $\mathrm{vec}(\cdot)$ does not preserve algebraic operations on its matrix inputs, \eg, $\mathrm{vec}(AB)$ is not a simple function of $\mathrm{vec}(A)$ and $\mathrm{vec}(B)$; as such, we use double indices to maintain the \textit{matrix} structure of points on $\stableK$.}
{A vector field $V$ on $\stableK$ is a smooth map $V\colon \stableK \mapsto T\stableK$, usually written as $K \mapsto V_K$, with the property that $V_K \in T_K \stableK$ for all $K \in \stableK$. A covariant 2-tensor field is a smooth real-valued multilinear function of 2 vector fields.
We denote the set of all vector fields over $\stableK$ by $\fX(\stableK)$, and the bundle of covariant 2-tensor fields on $\stableK$ by $T^2(T^*\stableK)$.}
Finally, for any general mapping $P\colon\stableK \mapsto \star$, we use $P_K, P|_K$ or $P(K)$ to denote the element in $\star$ that $K \in \stableK$ has been mapped to.
The following is a frequently used technical lemma.
\begin{lemma}\label[lemma]{lem:dlyap}
The subset $\Astable$ is an open submanifold of $\Amatrices$, the Lyapunov map $\lyap\colon \Astable \times \Amatrices \mapsto \Amatrices$ is smooth, and its differential acts as 
\begin{gather*}
    \diff \lyap_{(A,Q)}[E,F] =
    \lyap \big(A, E \lyap(A,Q) A^\intercal + A \lyap(A,Q) E^\intercal + F \big)
\end{gather*}
on any $(E,F) \in T_{(A,Q)} (\Astable \times \Amatrices)$ with the identification that follows by $T_{(A,Q)} (\Astable \times \Amatrices) \cong T_A \Astable \oplus T_Q \Amatrices \cong \Amatrices \oplus \Amatrices$.
Furthermore, for any $A \in \Astable$ and $Q, \Sigma \in \Amatrices$ we have, the so-called \emph{\lyaptrace} property,
\[\tr{\lyap(A^\intercal,Q) \Sigma} = \tr{\lyap(A, \Sigma) Q}.\]
\end{lemma}

Next, we note (see  \cref{eq:constained-lqr-opt2} in \S\ref{sec:application}) that many optimal control problems such as \ac{slqr}, \ac{olqr} and even \ac{lqg} share a similar cost structure, as
\(\textstyle f(K) = \frac{1}{2} \tr{P_K \Sigma_K},\)
where mappings $P, \Sigma\colon \stableK \mapsto \Amatrices$ send $K$ to
\begin{equation}
    P_K \coloneqq \lyap(A_{\mathrm{cl}}^\intercal, Q + K^\intercal R K),\quad
    \Sigma_K \coloneqq \Sigma_1 + K^\intercal \Sigma_2 K, \label{eq:sigmaK}
\end{equation}
respectively, with the closed-loop system $A_{\mathrm{cl}} \coloneqq A + B K$, and $\Sigma_1,\Sigma_2 \succeq 0$ as prescribed matrices with appropriate dimensions.
By the \lyaptrace property, the cost can be recast as $f(K) = \frac{1}{2} \tr{(Q + K^\intercal R K)\lyap(A_{\mathrm{cl}}, \Sigma_K)}$.
Motivated by this, we define a covariant 2-tensor field on $\stableK$ which will subsequently be proved to be a Riemannain metric. 
Riemannian metrics have been used in the literature to efficiently capture the geometry of the problem; \eg,~see \cite{absil2004cubically} for optimizing the Rayleigh quotient on the Grassmann manifold, and \cite{muller2023geometry} for natural policy gradient on Markov decision processes. Note that our metric is different from the ``Hessian metric'' induced by, say, a convex function \cite{alvarez2004hessian}.

\begin{lemma}\label[lemma]{lem:tensor}
Let $\tensor{\cdot}{\cdot}{Y} \colon \fX(\stableK) \times \fX(\stableK) \mapsto C^\infty(\stableK)$ denote the mapping that, under the usual identification of the tangent bundle, for any $V, W \in \fX(\stableK)$ sets,\footnote{The notation $\tensor{\cdot}{\cdot}{Y}$ should not be confused with the (ordinary) inner product in inner-product spaces as it is varying over $\stableK$.}
\begin{equation*}
    \tensor{V}{W}{Y}\big|_K \coloneqq \tr{(V_K)^\intercal \; W_K \; \lyap(A_{\mathrm{cl}},\Sigma_K)}, \quad \forall K \in \stableK.
\end{equation*}
Then this map, induced by a smooth symmetric covariant 2-tensor field,
is well-defined.
\end{lemma}

Now, let $\metric \colon \stableK \mapsto T^2(T^*\stableK)$ be a smooth section of the bundle $T^2(T^*\stableK)$ that sends $K$ to $\tensor{\cdot}{\cdot}{Y}\big|_K$. Then, $\metric$ is in fact a Riemannian metric under mild conditions formalized below.

\begin{proposition}\label[proposition]{prop:Riemannian-metric}
If $(A_{\mathrm{cl}}, \Sigma_K^{1/2})$ is controllable and 
for all $K \in \stableK$, $\Sigma_K~\succeq~0$, 
then $(\stableK,\metric)$ is a Riemannian manifold. Moreover,  if we define the mapping $Y \colon \stableK \mapsto \Amatrices$ sending $K$ to 
\[Y_K \coloneqq \lyap(A_{\mathrm{cl}}, \Sigma_K),\] 
then, with respect to the dual coframe $(d x^{i,j})$, $\metric = \metric_{(i,j) (k,\ell)} d x^{i,j} \otimes d x^{k, \ell}$, where each $\metric_{(i,j) (k,\ell)} \in C^\infty(\stableK)$ satisfies $\metric_{(i,j) (k,\ell)}(K) = [Y_K]_{\ell,j}$ if $i = k$, and $0$ otherwise.
Furthermore, the inverse \emph{matrix} $\metric^{(i,j) (k,\ell)}$ satisfies $\metric^{(i,j) (k,\ell)}(K) = [Y_K^{-1}]_{\ell,j}$ if $i = k$, and $0$ otherwise.
\end{proposition}

\begin{remark}
The premise of \Cref{prop:Riemannian-metric} is satisfied if $\Sigma_K \succ 0$ for all $K \in \stableK$; \eg, when $\Sigma_1 \succ 0$ and $\Sigma_2 \succeq 0$. 
Also, by some algebraic manipulations, the well-known Hewer's algorithm \cite{hewer1971iterative} can be viewed as a Riemannian quasi-Newton iteration with respect to this Riemannian metric but with the Euclidean connection--also see \Cref{rem:hewer}. 
Finally, we remark that the developed machinery can be generalized by modifying the Riemannian metric to incorporate a ``preconditioning'' positive definite $\Theta \succ 0$ such as:
\begin{equation*}
    \tensor{V}{W}{Y}\big|_K \coloneqq \tr{(V_K)^\intercal \;\Theta\; W_K \; \lyap(A_{\mathrm{cl}},\Sigma_K)}, \quad \forall K \in \stableK.
\end{equation*}
In this case, a judicious choice of $\Theta$ may, in general, improve the performance of corresponding numerical schemes. This metric also has a system theoretic interpretation that will be elaborated upon in our subsequent works.
\end{remark}

\subsubsection{Riemannian connection on $T\stableK$}
First, consider a Riemannian submanifold $(\substableK, \submetric)$ with $\submetric \coloneqq \iota^*_{\footnotesize \substableK} \metric$, where $\iota^*_{\footnotesize \substableK}$ denotes the pull-back by inclusion.
In order to understand the second order behavior--\ie,  the \emph{Hessian}--of a smooth function on $\substableK$, we need to study the notion of {\em connection} in $T \stableK$, and how 
it relates to analogous construct 
in the tangent bundle $T \substableK${--see \cite{lee2018introduction} for further details}.
Recall that by Fundamental Theorem of Riemannian Geometry, there exists a unique connection $\nabla \colon \fX(\stableK)\times \fX(\stableK) \mapsto \fX(\stableK)$ in $T \stableK$ that is compatible with $\metric$ and symmetric, \ie, for all $U,V,W \in \fX(\stableK)$ we have:
\begin{itemize}
    \item $\nabla_U \tensor{V}{W}{Y} = \tensor{\nabla_U V}{W}{Y} + \tensor{V}{\nabla_U W}{Y}$,
    \item $\nabla_U V - \nabla_V U \equiv [U,V]$,
\end{itemize}
where $[V,W] \in\fX(\stableK)$ denotes the Lie bracket of $V$ and $W$. 

Note that, the restriction of $\nabla$ to $\fX(\substableK) \times \fX(\substableK)$ would not be a connection in $T \substableK$ as its range does not necessary lie in $\fX(\substableK)$.
However, we can denote the (Riemannian) \textit{tangential} and \textit{normal projections} by $\tproj\colon T \stableK|_{\footnotesize\substableK} \mapsto T \substableK$ and $\nproj\colon T \stableK|_{\footnotesize\substableK} \mapsto N \substableK$, respectively, with $N \substableK$ indicating the normal bundle of $\substableK$. Then, by  Gauss Formula, if $\widetilde{\nabla}\colon\fX(\substableK) \times \fX(\substableK) \mapsto \fX(\substableK)$ denotes the Riemannian connection in the tangent bundle $T \substableK$, computed via,
\begin{equation}\label{eq:subconnection}
    \widetilde{\nabla}_U V = \tproj \nabla_U V,
\end{equation}
for any $U,V \in \fX(\substableK)$  arbitrarily extended to vector fields on a neighborhood of $\substableK$ in $\stableK$.

{For computational purposes, we also obtain the Christoffel symbols associated with $\metric$ (denoted by $\Gamma^{(i,j)}_{(k,\ell)(p,q)}$) in the global coordinate frame. This would, in turn, completely characterize the connection $\nabla$ and facilitates its computation in this frame.}
\vspace{-0.1cm}
\begin{proposition}\label[proposition]{lem:christoffel}
Consider a point $K \in \stableK$ and, under the usual identification of $T \stableK$, define \footnote{This coincides with the action of $\partial_{(p,q)}|_K$ on the mapping $K \mapsto Y_K$.}
\begin{gather*}
    \begin{aligned}
        \lyapij_K{(p,q)} \coloneqq \lyap \big(A_{\mathrm{cl}},\; &B \partial_{(p,q)} Y_K A_{\mathrm{cl}}^\intercal + A_{\mathrm{cl}} Y_K \partial_{(p,q)}^\intercal B^\intercal\\
    &+ \partial_{(p,q)}^\intercal \Sigma_2 K + K^\intercal \Sigma_2 \partial_{(p,q)} \big)
    \end{aligned}
\end{gather*}
for each $(p,q) \in [m] \times [n]$ where $Y_K = \lyap(A_{\mathrm{cl}},\Sigma_K)$.
Then, the Christoffel symbols associated with the metric $\metric$ in the global coordinate frame $(\partial_{(i,j)})$ satisfies $\Gamma^{(i,j)}_{(k,\ell)(p,q)} (K) =$
\begin{gather*}
    \begin{cases}
    \sfrac{1}{2}\left[\lyapij_K{(p,q)}\, Y_K^{-1} \right]_{(\ell,j)},  & \text{if $k = i \neq p$,}\\
    \sfrac{1}{2} \left[ \lyapij_K{(k,\ell)}\, Y_K^{-1}\right]_{(q,j)}, & \text{if $p = i \neq k$,}\\
    \textstyle
    - \sfrac{1}{2}\sum_{s} \left[ \lyapij_K{(i,s)}  \right]_{(q,\ell)} [Y_K^{-1}]_{(s,j)}, & \text{if $p = k \neq i$,}\\
    \textstyle \sfrac{1}{2}\sum_{s} \big( \left[\lyapij_K{(i,\ell)}\right]_{(q,s)} + \left[\lyapij_K{(i,q)}\right]_{(\ell,s)} & \\
    \qquad\quad -\left[\lyapij_K{(i,s)}\right]_{(q,\ell)} \big) \;[Y_K^{-1}]_{(s,j)}, 
    & \text{if $p = k = i$,}\\
    % 0 & \text{if $k \neq i \neq p \neq k$.}\\
    0, & \text{otherwise.}
    \end{cases}
\end{gather*}
\end{proposition}

\begin{remark}\label{remark:geodesics}
Note that with respect to the global coordinates of $(\stableK,\metric,\nabla)$, the \textit{geodesic equation} is a system of $(mn)$ second-order ordinary differential equations whose varying coefficients involve $(mn)^3$ Christoffel symbols $\Gamma^{(i,j)}_{(i,\ell)(i,q)}$ as obtained above. Therefore, computing the Riemannian Exponential mapping is computationally burdensome; as such,
in this work, we avoid using it as a retraction.
\end{remark}

\subsection{Extrinsic analysis of a smooth function constrained on a Riemannian submanifold}\label{sec:extrinsic}
In this subsection, we study the gradient and Hessian operators of a constrained smooth function from an extrinsic point of view, which is yet to be defined. In other words, we consider $(\substableK,\submetric)$ as a Riemannian submanifold of $(\stableK,\metric)$, where $\submetric = \iota_{\footnotesize\substableK}^* \metric$, with $\iota_{\footnotesize\substableK}^*$ denoting the pull-back by inclusion of $\substableK$ into $\stableK$. Then, by considering any smooth function $f$ on $\stableK$, we can define its restriction to $\substableK$ as
\[h \coloneqq f|_{\footnotesize\substableK},\]
and examine how its gradient and Hessian operators are related to those of $f$. In order to answer this question, we utilize the Riemannian connection to analyze the second order behavior of $f$ (or that of $h$). 

First, recall from \cite{lee2018introduction} that the gradient of $f$ with respect to the Riemannian metric $\metric$, denoted by $\grad{f} \in \fX(\stableK)$, is the unique vector field satisfying 
\[\tensor{V}{\grad{f}}{Y} = V f, \]
for any $V \in \fX(\stableK)$.
Then, we denote the \emph{Hessian operator} of $f \in C^\infty(\stableK)$ as the map $\hess{f}\colon\fX(\stableK) \mapsto \fX(\stableK)$ defined by
\[\hess{f}[U] \coloneqq \nabla_U \grad{f},\]
for any $U \in \fX(\stableK)$. Note that we use the same notation to denote the gradient and Hessian operators defined on the submanifold $\substableK$
(see the appendix).
%\S\ref{app:diffgeom}. 
Finally, for any normal vector field $N$ (\ie, a smooth section of $N\substableK$), the \emph{Weingarten map} in the direction of $N$ is a self-adjoint linear map denoted by $\wein_N\colon \fX(\substableK) \mapsto \fX(\substableK)$, {which defines a smooth bundle homomorphism from $T\substableK$ to itself (linear on each tangent space), characterized by \cite{lee2018introduction},
\[\tensor{\wein_N[V]}{W}{Y} = \tensor{N}{\nproj(\nabla_V W)}{Y}, \quad \forall V,W \in \fX(\stableK).\]}%
Now, we can formalize this abstract extrinsic analysis as follows.

\begin{proposition}\label[proposition]{prop:extrinsic}
Suppose $\substableK$ is an embedded Riemannian submanifold of $\stableK$, both equipped with their respective Riemannian connections. Let $f \in C^{\infty}(\stableK)$ be any smooth function; then $h \coloneqq f |_{\footnotesize\substableK}$ is smooth on $\substableK$ and we have
\[\grad{h} = \tproj (\grad{f}|_{\footnotesize\substableK}).\]
Furthermore, under the usual identification of $T\substableK \subset T\stableK$, for any $V \in \fX(\substableK)$ we have,
\begin{equation*}
    \hess{h}[V] = \tproj (\hess{f}[V]\big|_{\footnotesize\substableK}) + \wein_{\nproj (\grad{f}|_{\footnotesize\substableK})}[V],
\end{equation*}
where $V$ is arbitrarily extended to vector fields on a neighborhood of $\substableK$ in $\stableK$.
\end{proposition}

\subsubsection{On the choice of connection}\label{sec:choice_connection}
On the manifold $(\stableK,\metric,\nabla)$, computing the exponential map requires finding solution to a system of
ordinary differential equations 
 %\ac{ode}s 
of dimension $m n$. 
This approach is not only computationally demanding but also does not necessarily provide an exponential map of the submanifold $\substableK$ (unless it happens to be totally geodesic). 
In order to avoid the computation of the Riemannian exponential map, it seems reasonable to perform updates by using simpler \emph{retractions} from the tangent bundle to the manifold (cf. \cite{absil2009optimization}); however, in general, we do not have access to such a retraction in our setup. Another computational overhead of utilizing the Riemannian connection associated with the Riemannian metric $\metric$ pertains to the $(mn)$-number of Lyapunov equations involved in obtaining the Christoffel symbols at each point. 

On the other hand, for applications in which the submanifold appears as  $\substableK = \stableK \cap \constraint$, where $\constraint$ is an affine subspace of $\Kmatrices$, it might seem reasonable to consider the ambient manifold $(\stableK,\metric, \Enabla)$, where $\Enabla$ refers to the so-called \emph{Euclidean connection}, \ie, the connection whose symbols (with respect to the global coordinates) all vanish (\ie, $\overline{\Gamma}^{i,j}_{(k,\ell)(p,q)} \equiv 0$ on $\stableK$). This results in a simpler \emph{Hessian} operator which, however, does not respect the geometry of $(\stableK,\metric)$ simply because $\Enabla$ is not compatible with the metric $\metric$--in contrast to its associated Riemannian connection.
Nonetheless, for completeness, we also define the \emph{Euclidean Hessian operator} of $f \in C^\infty(\stableK)$ as the map $\euchess{f}\colon\fX(\stableK) \mapsto \fX(\stableK)$ defined by
\[\euchess{f}[U] \coloneqq \Enabla_U \grad{f},\]
for any $U \in \fX(\stableK)$. This operator enjoys similar properties as that of $\hess{f}$, but contains different second order information about $f$ (\eg, see \Cref{fig:motivation} for a comparison). 
%

% %%%%%%%%%%%%%%%%%%%%%%%%%%%%%%%%%%%%%%%%%%%%%%%%%%%%%%%%%%%%%%%%%
\section{Riemannian Optimization on Submanifolds of $\stableK$ with Linear Structure} \label{sec:learning-algo}
In this section, we propose an optimization
algorithm for 
%for learning local optima of 
smooth cost functions, constrained to submanifolds of $\stableK$ that are endowed with a linear structure; that is, $\substableK =\stableK \cap \constraint$, where $\constraint$ entails a linear structure in $\Kmatrices$. The proposed algorithm,
does not involve the exponential mapping (due to its computational complexity); note that no other retraction from the tangent space onto the manifold $\stableK$ is known. Instead, we exploit this linear structure together with a geometrically-induced stability certificate that guarantees stability of the iterates by adjusting the respective 
stepsize.

In what follows, we first introduce this
stability certificate and then propose 
the algorithm.
We then show how this certificate can 
be utilized to choose stepsizes that
guarantee a linear convergence rate;
furthermore, we will discuss existence of neighborhoods--containing a local minima--on which the algorithm achieves a quadratic rate of convergence.

\subsubsection{Stability certificate and (direct) policy optimization}
Recall that $\stableK$ is open in $\Kmatrices$. Nonetheless, we provide the following result
that quantifies this fact with respect to the problem parameters; an observation that
has an immediate utility for 
analyzing iterative algorithms on $\stableK$. 

\begin{lemma}\label[lemma]{lem:stability-cert}
Consider a smooth mapping $\mathcal{Q}\colon \stableK \mapsto \Amatrices$ that sends $K$ to any $\mathcal{Q}_K \succ 0$.
For any direction $G \in T_K\stableK \cong \Kmatrices$ at any point $K \in \stableK$, if 
\begin{gather*}
    0 \leq \eta \leq \stabcer_K \coloneqq \Large\sfrac{\lambdamin(\mathcal{Q}_K)}{\left(2 \lambdamax\left(\lyap(A_{\mathrm{cl}}^\intercal, \mathcal{Q}_K)\right) \|BG\|_2 \right)},
\end{gather*}
then $K^+ \coloneqq K + \eta G \in \stableK$; $\stabcer_K$ will be referred to as the \emph{stability certificate} at $K$.
\end{lemma}
\begin{proof}
As $K$ is stabilizing, for any such $\mathcal{Q}_K$, there exists a matrix $P = \lyap(A_{\mathrm{cl}}^\intercal, \mathcal{Q}_K) \succ 0$ satisfying $P = A_{\mathrm{cl}^+}^\intercal P A_{\mathrm{cl}^+} + L$
with $A_{\mathrm{cl}^+} \coloneqq A + B K^+$ and
$L \coloneqq  \mathcal{Q}_K + A_{\mathrm{cl}}^\intercal P A_{\mathrm{cl}} - A_{\mathrm{cl}^+}^\intercal P A_{\mathrm{cl}^+}$.
Therefore, in order to establish that $K^+$ is stabilizing, as $P \succ 0$--by the Lyapunov Stability Criterion \cite[Theorem8.4]{Hespanha2018linear}--it suffices to show that $L\succ 0$. 
Next,
\begin{align}
L =&  \mathcal{Q}_K - \eta G^\intercal B^\intercal P A_{\mathrm{cl}} - \eta A_{\mathrm{cl}}^\intercal P B G - \eta^2 G^\intercal B^\intercal P B G \nonumber\\
  \succcurlyeq &  \mathcal{Q}_K - a A_{\mathrm{cl}}^\intercal P A_{\mathrm{cl}} - (1 + {\large\sfrac{1}{a}})\eta^2 G^\intercal B^\intercal P B G \nonumber\\
  =& (1+a)\mathcal{Q}_K - a P - (1+{\large\sfrac{1}{a}})\eta^2 G^\intercal B^\intercal P B G, \label{eq:L-lowerbound}%\\
%   =& (1+2a)\mathcal{Q}_K/2 - a P + \mathcal{Q}_K/2 - (1+{\large\sfrac{1}{a}})\eta^2 G^\intercal B^\intercal P B G
\end{align}
because for any $a>0$, $P \succ 0$ implies
\begin{gather*}
a A_{\mathrm{cl}}^\intercal P A_{\mathrm{cl}} +  ({\large\sfrac{\eta^2}{a}}) G^\intercal B^\intercal P B G \succcurlyeq
     \eta G^\intercal B^\intercal P A_{\mathrm{cl}} + \eta A_{\mathrm{cl}}^\intercal P B G.
\end{gather*}
Now, by recalling the infinite-sum representation of $P$ and the fact that $\mathcal{Q}_K\succeq 0$, we conclude that $\lambdamax(P) \geq \lambdamax(\mathcal{Q}_K) \geq \lambdamin(\mathcal{Q}_K)$. Then, 
we proceed as follows: if $\lambdamax(P) > \lambdamin(\mathcal{Q}_K)$ then we choose $a = \Large\sfrac{\lambdamin(\mathcal{Q}_K)}{(2\lambdamax(P)- 2\lambdamin(\mathcal{Q}_K))} > 0$;
% \[1 + 2a = {\lambdamax(P)}/{(\lambdamax(P)- \lambdamin(\mathcal{Q}_K))};\]
% or equivalently:
% \[1 + {1}/{a} = 2{\lambdamax(P)}/{\lambdamin(\mathcal{Q}_K)} - 1;\]
otherwise, we choose $a > 2\eta^2 \|BG\|_2^2$.
Either way, by comparing the minimum eigenvalues of both sides in \cref{eq:L-lowerbound},
\begin{gather*}
    \lambdamin(L) \geq \lambdamin(\mathcal{Q}_K)/2 
    - \left[{\large\sfrac{2\lambdamax(P)^2}{\lambdamin(\mathcal{Q}_K)}} - \lambdamax(P)\right] \eta^2 \|B G\|_2^2.
\end{gather*}
Therefore, if $|\eta| \leq \Large\sfrac{\lambdamin(\mathcal{Q}_K)}{(2\lambdamax(P) \|B G\|_2)}$,
then $L \succ 0$, hence completing the proof.
\end{proof}

\begin{remark}
The proceeding lemma also provides a \emph{conditioning} of the optimization problem in terms of system parameters $A, B$. In other words, for any choice of $\mathcal{Q}_K\succ 0$ at any $K \in \stableK$, the ratio ${\lambdamax\left(\lyap(A_{\mathrm{cl}}^\intercal,\mathcal{Q}_K)\right)}/{\lambdamin(\mathcal{Q}_K)}$ represents a condition number revealing geometric information on the manifold at $K$. In a sense, this ratio reflects the \emph{Riemannian curvature} of $(\stableK,\metric,\nabla)$;
this connection will be further explored in 
our future work.
\end{remark}

Next, we propose an algorithm with convergence guarantees with at least a linear rate (when the iterates are far from the local optima) and eventually a Q-quadratic rate (when the iterates are close enough to the local optima).
The complication here is that we do not have access to a retraction with a reasonable computational complexity (see \Cref{remark:geodesics}).
We claim that, starting close enough to a local minimum, a Newton-type method using Riemannian metric and the Euclidean/Riemannian connection must converge quadratically if one could have used the stepsize $\eta = 1$. This is in fact due to the exponential mapping with respect to the Euclidean connection that serves as a retraction with desirable properties.
However, the stability certificate suggests that at least away from the local minimum, it might not be possible to use such a large stepsize. Therefore, a stepsize rule has to be deduced--that in turn, hinges upon the stability certificate; the resulting algorithm is summarized in \Cref{alg:Controller}. 
Hereafter, we refer to the solution $G \in T_K \substableK$ of the following equation as the \textit{Newton direction} on $\substableK$:
\begin{gather*}
    \hess{h}_K[G] = -\grad{h}_K,
\end{gather*}
where $h = f|_{\footnotesize\substableK}$; similarly, 
when $\hess{h}$ is replaced by $\euchess{h}$,
the corresponding solution is referred to as the \textit{Euclidean Newton direction}.

\begin{algorithm}[!ht]
	\caption{\acf{ouralgo} for Constrained Problems on $\stableK$}
	\algorithmfootnote{{For examples of the mapping $\mathcal{Q}$ in Line 2 see \Cref{rem:Q-mapping}.} In Line 4, $\hess{h}$ can be replaced by its Euclidean counterpart $\euchess{h}$. The update in Line 7 is possible due to the linear structure of $\substableK$ induced by $\constraint$. For different choices of stopping criteria see \cite{talebi2022geometric}.\vspace{-0.2cm}}
	\begin{algorithmic}[1]
		\State \textbf{Initialization:} Problem parameters $(A,B)$,\newline the linear constraint $\constraint$ and an initial feasible stabilizing controller $K_0 \in \substableK = \stableK \cap \constraint$
% 		\State \textbf{Output:} A locally optimal controller
		\State Choose a smooth mapping $K \mapsto \mathcal{Q}_K \succ 0$; set $t= 0$
		\State \textbf{Until stopping criteria are met, do}
		\State \hspace{5mm} Find the Newton direction $G_t$ on $\substableK$ satisfying
		    \[\hess{h}_{K_t}[G_t] = -\grad{h}_{K_t}\]
		\State \hspace{5mm} Use $\mathcal{Q}_{K_t}$ to obtain a stability certificate $\stabcer_{K_t}$ 
		\State \hspace{5mm} Compute step-size
		\(\eta_t = \min\left\{\stabcer_{K_t}, 1\right\} \)
		\State \hspace{5mm} Update:
		\(K_{t+1} = K_{t} + \eta_t G_t\)
		\State \hspace{5mm} $t \leftarrow t+1$
	\end{algorithmic}
	\label{alg:Controller}
\end{algorithm}

\subsubsection{Linear-quadratic convergence of \ac{ouralgo}}
In this section, we establish the local linear-quadratic convergence of \ac{ouralgo} on the submanifold $\substableK$ using differential geometric techniques \cite{lee2018introduction, absil2009optimization, gabay1982minimizing}. Herein, avoiding the exponential map induced by the Riemannian connection for updating the iterates, and 
instead relying the stability certificate,
adds another layer of complications for 
the convergence analysis.
To proceed, we say that $K^*$ is a critical point of $h$ if $\grad{h}_{K^*} = 0$; it is \emph{nondegenerate} if $\hess{h}_{K^*}$ is nondegenerate, \ie,
$\tensor{\hess{h}_{K^*}[G_1]}{G_2}{Y_{K^*}} = 0, \, \forall G_2 \in T_{K^*}\substableK$ implies that $G_1 = 0 \in T_{K^*}\substableK.$

\begin{lemma}\label[lemma]{lem:nondegenerate}
Suppose $K^*$ is a nondegenerate local minimum of $h \coloneqq f|_{\footnotesize\substableK}$. Then, it is isolated, $\grad{h}_{K^*} = 0$, and there exists a neighborhood of $K^*$ on which $\hess{h}$ is positive definite. Furthermore, $\hess{h}_{K^*} = \euchess{h}_{K^*}$.
\end{lemma}

\begin{theorem}\label[theorem]{thm:convergence}
Suppose $K^*$ is a nondegenerate local minimum of $h \coloneqq f|_{\footnotesize\substableK}$ over the submanifold $\substableK = \stableK \cap \constraint$ for some linear constraint $\constraint$. Then, there exists a neighborhood $\nghbd^* \subset \substableK$ of $K^*$ with the following property: whenever $K_0 \in \nghbd^*$, the sequence $\{K_t\}$ generated by \ac{ouralgo} remains in $\nghbd^*$ (therefore, it is stabilizing), and it converges to $K^*$ at least at a linear rate--and eventually--with a quadratic one.
\end{theorem}

\begin{remark}
The above result implies that there exist neighborhoods containing each nondegenerate local minimum of the constrained cost function on which the convergence of \ac{ouralgo} is guaranteed. The usefulness of this result is that the initial iterate $K_0$ is not required to be in a (small) neighborhood of the optimum on which the step-size $\eta = 1$ is feasible. Instead, by carefully incorporating the stability certificate (\Cref{lem:stability-cert}), we can obtain a larger basin of attraction for the iterates (see \Cref{fig:landscape-olqr} in \S\ref{sec:simulation}). Finally, even though the convergence rate is initially linear, as the algorithm proceeds,
a quadratic convergence rate is achieved.
\end{remark}

\vspace{-0.4cm}
\section{Feedback Synthesis via \ac{ouralgo}}\label{sec:application}
In this section, we discuss applications of 
the developed methodology for optimizing the \ac{lqr} cost over two distinct submanifolds, namely
those induced by \acf{slqr} and \acf{olqr} problems.
Consider a discrete-time linear time-invariant dynamics 
\begin{align}\label{eq:dynamics}
    \x_{k+1} = & A \x_k + B \uu_k, \qquad \y_k = C \x_k,
\end{align}
where $A \in \Amatrices$, $B \in \Bmatrices$ and $C \in \Cmatrices$ are the system parameters for some integers $n$, $m$ and $d$; $\x_k$, $\y_k$ and $\uu_k$ denote the states, output and input vectors at time $k$, respectively, and $\x_0$ is given. Conventionally, the \acf{lqr} problem is to find the sequence $\uu = (\uu_k)_0^\infty \in \ell_2$ that minimizes the following quadratic cost
\begin{equation}\label{eq:costx0}
     \textstyle J_{\x_0}(\uu) = \sfrac{1}{2}\sum_{k=0}^\infty \x_k^\intercal Q \x_k +  \uu_k^\intercal R \uu_k,
\end{equation}
subject to \cref{eq:dynamics}, where $Q = Q^\intercal \succcurlyeq 0$ and $R = R^\intercal \succ 0$ are prescribed cost parameters. 
It is well known (see \eg, \S22.7 in \cite{goodwin2001control}) that the optimal state-feedback solution $\uu^*$ to this problem reduces to solving the \ac{dare} for the optimal cost matrix $P_{\mathrm{LQR}}$. This results in a linear state-feedback optimal control $\uu_k^*= K_{\mathrm{LQR}} \x_k$, where $K_{\mathrm{LQR}} \in \Kmatrices$ is the optimal \ac{lqr} gain (policy) obtained from $P_{\mathrm{LQR}}$. Furthermore, the associated optimal cost can be obtained as $J_{\x_0}(\uu^*) = \sfrac{1}{2} \, \x_0^\intercal P_{\mathrm{LQR}} \x_0$.

Naturally, one could think of the \ac{lqr} cost as a map $K \mapsto J_{\x_0}(\uu)|_{\uu=K \x}$; however, this would depend on $\x_0$ and generally, its value can still be finite while $K$ is not necessarily stabilizing (\ie, when $K \notin \stableK$). 
Instead, in order to avoid the dependency on the initial state while considering the constraints on the policy directly, we pose the following constrained optimization problem,
\vspace{-0.1cm}
\begin{align}\label{eq:constained-lqr-opt}
    \min_K f(K) & \coloneqq \textstyle \mathop{\mathbb{E}}_{\x_0 \sim \mathcal{D}}  J_{\x_0}(\uu) \\
   \text{s.t.}~~ \x_{k+1} &= A \x_k + B \uu_k, \; \uu_k = K \x_k, \, \forall k \geq0, \; K \in \substableK, \nonumber \vspace{-0.2cm}
\end{align}
where $\substableK$ is an embedded submanifold of $\stableK$, and $\mathcal{D}$ denotes a distribution of zero-mean multivariate random variables of dimension $n$ with covariance matrix $\Sigma_1$ so that $0 \prec \Sigma_1 = \Sigma_1^\intercal \in \Amatrices$.

Next, we can reformulate \cref{eq:constained-lqr-opt} as follows. For each stabilizing controller $K \in \stableK$, from \cref{eq:dynamics} and \cref{eq:costx0} we have that
\begin{gather*}
    J_{\x_0}(K\x) \textstyle =\frac{1}{2} \sum_{k=0}^\infty \x_0^\intercal (A_{\mathrm{cl}}^k)^\intercal [Q + K^\intercal R K] A_{\mathrm{cl}}^k \x_0,\vspace{-0.2cm}
\end{gather*}
where $A_{\mathrm{cl}} \coloneqq A + B K$. Since, $A_{\mathrm{cl}}$ is a stability matrix, the sum $\sum_{k=0}^\infty (A_{\mathrm{cl}}^k)^\intercal [Q + K^\intercal R K] A_{\mathrm{cl}}^k$ converges, which is equal to the unique solution $ P_K \coloneqq \lyap(A_{\mathrm{cl}}^\intercal, K^\intercal R K + Q)$.
Therefore, $f(K) =\frac{1}{2} \mathop{\mathbb{E}}_{\x_0 \sim \mathcal{D}} \tr{P_K \x_0 \x_0^\intercal} = \frac{1}{2} \tr{P_K \Sigma_1}$, and thus the problem in \cref{eq:constained-lqr-opt} reduces to
\begin{equation}\label{eq:constained-lqr-opt2}
    \min_K f(K)  = \textstyle \frac{1}{2} \tr{P_K \Sigma_1} \quad  \text{s.t. } \quad K \in \substableK.
\end{equation}

This reformulation of the \ac{lqr} cost function has been 
previously adopted in the literature (see \eg, \cite{fazel2018global, martensson2012gradient, bu2019lqr}) but the inherent geometry of the submanifold $\substableK$ has
generally been overlooked.
In the absence of constraints, \ie, $\substableK = \stableK$, Hewer's algorithm is known to converge to the optimal state feedback gain at a quadratic rate \cite{hewer1971iterative}, given controllability of $(A,B)$ and stability of the initial controller $K_0$--see \cite{talebi2021regularizability} for further study regarding the initial controller.
Otherwise, $\substableK$ may have disconnected components, and in general, the constrained cost function may have stationary points that are not local minima.
Nonetheless, in this section, we apply the techniques developed in \S\ref{sec:results} and  \Cref{alg:Controller} to the constraint arising in the well-known \ac{slqr} and \ac{olqr} problems. Note that both of these problems can be cast as an optimization in \cref{eq:constained-lqr-opt2} with $\substableK$ denoting a specific submanifold of $\stableK$ that will be further discussed in \S\ref{subsec:slqr} and \S\ref{subsec:olqr}, respectively.

\subsection{Solving for the Newton direction}\label{subsec:newton-direction}
In order to solve for the Newton direction at any $K \in \substableK$, suppose that the tuple $(\widetilde{\partial}_{(p,q)}|_{(p,q) \in D})$ denotes a smooth local frame for $\substableK$ on a neighborhood of $K$, where $D$ is a subset of $[m]\times [n]$ depending on the dimension of $\substableK$.\footnote{Each $T_K \substableK$ can be viewed as a subspace of $T_K \stableK$ as $\substableK \subset \stableK$ is embedded.}
In fact, by \Cref{prop:extrinsic}, the Newton direction $G= [G]^{k,\ell} \widetilde{\partial}_{(k,\ell)}|_K \in T_K\substableK$ (interpreted as a subspace of $T_K \stableK$) can be computed by solving the following system of $|D|$-linear equations (for each index $(p,q) \in D$),
\begin{gather*}
    \sum_{(k,\ell) \in D} [G]^{k,\ell} \; h_{;(k,\ell)(p,q)}(K) = - \tensor{\tproj(\grad{f}|_K)}{\widetilde{\partial}_{(p,q)}|_K}{Y_K},
\end{gather*}
where $h_{;(k,\ell)(p,q)}$ denote the coordinates of $\widetilde{\nabla}^2 h$ with respect to the local coframe dual to $(\widetilde{\partial}_{(p,q)}|_{(p,q) \in D})$ (in tensor notation; \eg, see \cite[Example 4.22]{lee2018introduction}). Thus, by \cref{eq:covariantHessian}, $h_{;(k,\ell)(p,q)}(K) = \tensor{\hess{h}_K[\widetilde{\partial}_{(k,\ell)}|_K]}{\widetilde{\partial}_{(p,q)}|_K}{Y_K};$
or with $\hess{h}$ replaced by $\euchess{h}$, depending on the connection.

\subsubsection{Analysis for the special cost function}
We now turn our attention towards the analysis of the following cost function specific to optimal control problems--see \cite{talebi2022policy} for proofs of the results in this subsection. In order to specialize 
the results obtained so far to this case, we set $\Sigma_2 = 0$ in \cref{eq:sigmaK}.

\begin{proposition}\label[proposition]{prop:lqrcost}
On the Riemannian manifold $(\stableK,\metric,\nabla)$, define $f \in C^\infty(\stableK)$ with $f(K) = \frac{1}{2} \tr{P_K \Sigma_1}$, where $P_K = \lyap(A_{\mathrm{cl}}^\intercal, K^\intercal R K + Q)$. Then, $f$ is smooth and under the usual identification of the tangent bundle 
\begin{gather*}
    \grad{f}_K = R K + B^\intercal P_K A_{\mathrm{cl}}.
\end{gather*}
Furthermore, $\hess{f}$ and $\euchess{f}$ are both self-adjoint operators such that, for any $E,F \in T_K \stableK$,
\begin{align*}
    &\tensor{\hess{f}_K[E]}{F}{Y_K}= \tensor{ B^\intercal (S_K[F]) A_{\mathrm{cl}}}{E}{Y_K}\\
    &\qquad\qquad +\tensor{(R+B^\intercal P_K B) E + B^\intercal (S_K[E])A_{\mathrm{cl}}}{F}{Y_K} \\
    &\qquad\qquad - \tensor{\grad{f}_K}{[E]^{k,\ell}\, [F]^{p,q}\; \Gamma^{i,j}_{(k,\ell)(p,q)}(K) \partial_{i,j}}{Y_K},\\
    & \tensor{\euchess{f}_K[E]}{F}{Y_K} = \tensor{ B^\intercal (S_K[F]) A_{\mathrm{cl}}}{E}{Y_K}\\
    &\qquad\qquad +\tensor{(R+B^\intercal P_K B) E + B^\intercal (S_K[E])A_{\mathrm{cl}}}{F}{Y_K},
\end{align*}
with $\Gamma^{i,j}_{(k,\ell)(p,q)}$ denoting the Christoffel symbols of $g$ and
\begin{equation*}
    S_K[E] \coloneqq \lyap(A_{\mathrm{cl}}^\intercal, E^\intercal \grad{f}_K + (\grad{f}_K)^\intercal E).
\end{equation*}
\end{proposition}
\begin{remark}\label{rem:hewer}
{For comparison, in the absence of any constraint (\ie, when $\substableK=\stableK$), the Hewer's update $K_{t+1} = -(B^\intercal P_{K_t} B + R)^{-1}B^\intercal P_{K_t} A$ in \cite{hewer1971iterative} can be written as 
\[K_{t+1} = K_t + \widehat{G}_t,\] 
with the Riemannian quasi-Newton direction $\widehat{G}_t$ satisfying,
\[\widehat{H}_{K_t}[\widehat{G}_t] = - \grad{f}_{K_t}\]
where $\widehat{H}_{K_t}: = (R+B^\intercal P_{K_t} B)$ is a positive definite approximation of $\hess{f}_{K_t}$ and $\euchess{f}_{K_t}$.
The \emph{algebraic coincidence} is that the unit stepsize remains stabilizing throughout these quasi-Newton updates. In general, and particularly on constrained submanifolds $\substableK$, one needs to  instead utilize
the stability certificate developed in \Cref{lem:stability-cert}.}
\end{remark}

As a consequence of \Cref{prop:extrinsic}, the next corollary is the analogue of \Cref{prop:lqrcost} on any submanifold of $\stableK$.

\begin{corollary}\label[corollary]{cor:lqrcost-constrained}
Under the premise of \Cref{prop:lqrcost}, let $h = f|_{\footnotesize\substableK}$, where $\substableK \subset \stableK$ is an embedded Riemannian submanifold with the induced connection. Then, $h$ is smooth and under the usual identification of the tangent bundle 
\[\grad{h}_K = \tproj(R K + B^\intercal P_K A_{\mathrm{cl}}).\]
Furthermore, $\hess{h}$ is a self-adjoint operator and can be characterized as follows: for any $E,F \in T_K \substableK \subset T_K \stableK$,
\begin{multline*}
    \tensor{\hess{h}_K[E]}{F}{Y_K}= \tensor{ B^\intercal (S_K[F]) A_{\mathrm{cl}}}{E}{Y_K}\\
    +\tensor{(R+B^\intercal P_K B) E + B^\intercal (S_K[E])A_{\mathrm{cl}}}{F}{Y_K} \\
    - \tensor{\grad{h}_K}{[E]^{k,\ell}\, [F]^{p,q} \;\Gamma^{i,j}_{(k,\ell)(p,q)}(K) \partial_{i,j}}{Y_K},
\end{multline*}
where $\Gamma^{i,j}_{(k,\ell)(p,q)}$ are the Christoffel symbols of $g$ and
\begin{equation*}
    S_K[E] \coloneqq \lyap \big(A_{\mathrm{cl}}^\intercal, E^\intercal \grad{f}_K + (\grad{f}_K)^\intercal E \big).
\end{equation*}
\end{corollary}

\begin{remark}\label{rem:Q-mapping}
If $Q \succ 0$, then we can choose the mapping $\mathcal{Q}\colon K\mapsto \mathcal{Q}_K$ to be $\mathcal{Q}_K = Q + K^\top R K$, and thus the stability certificate $\stabcer_K$ as defined in \Cref{lem:stability-cert} satisfies,
\begin{gather*}
    \stabcer_K \geq \Large\sfrac{\lambdamin(Q) \lambdamin(\Sigma_1)}{(4 f(K) \, \|BG\|_2)},
\end{gather*}
where $f$ denotes the \ac{lqr} cost.
This is due to the fact that $R, Q, \Sigma_1 \succ 0$ and so is $ P_K \succ 0$; hence, by the trace inequality,
\(f(K) \geq (\sfrac{1}{2}) \lambdamin(\Sigma_1) \lambdamax(P_K).\)
The claimed lower-bound on the stability certificate then follows by combining the last inequality with the definition of the stability certificate.
Otherwise if $Q \succeq 0$, one can leverage observability of the pair $(A, Q^{1/2})$ to derive analogous results.
\end{remark}

\subsubsection{State-feedback \ac{slqr}} \label{subsec:slqr}
Any desired sparsity pattern on the controller gain $K$ imposes a linear constraint set, denoted by $\constraint_D$, which indicates a linear subspace of $\Kmatrices$ with nonzero entries only for a prescribed subset $D$ of entries, \ie, for any $K \in \constraint_D$ and $(i,j) \notin D$ we must have $[K]_{i,j}=0$. Then,
$\substableK = \stableK \cap \constraint_D$ is a properly embedded submanifold of dimension $|D|$.
Furthermore, at any point $K \in \substableK$ and for any tangent vector $E \in T_K \stableK$, we can compute the tangential projection $\tproj\colon T_K\stableK \mapsto T_K\substableK$ as the unique solution of 
\vspace{-0.1cm}
\begin{equation}\label{eq:tproj-linear}
    \eucproj_{\constraint_D}\big[(E-\widetilde{E}^*)Y_K \big]= 0,
\end{equation}
where $\eucproj_{\constraint_D}$ denotes the Euclidean projection onto the sparsity pattern $\constraint_D$. Note that at each $K \in \substableK$, the last equality consists of $|D|$ nontrivial linear equations involving $|D|$ unknowns (as the nonzero entries of $\widetilde{E}^*$), which can be solved efficiently.
Finally, if $\widetilde{\partial}_{(i,j)}$ (as described in \S\ref{subsec:newton-direction}) is taken to be $\widetilde{\partial}_{(i,j)} = \partial_{(i,j)}$ for $(i,j) \in D$, then $(\widetilde{\partial}_{(i,j)}|_{(i,j) \in D})$ forms a global smooth frame for $T\substableK$.
Thus, for each $(k,\ell),(p,q) \in D$, the coordinates $h_{;(k,\ell)(p,q)}(K)$ simplifies to
\vspace{-0.1cm}
\begin{align*}%\label{eq:cov_hess_h-coord-reduced}
    &\hspace{-0.9cm}h_{;(k,\ell)(p,q)}(K) = \tensor{ B^\intercal (S_K[\partial_{(p,q)}]) A_{\mathrm{cl}}}{\partial_{(k,\ell)}}{Y_K}\nonumber\\
    \hspace{0.9cm}&+\tensor{(R+B^\intercal P_K B) \partial_{(k,\ell)} + B^\intercal (S_K[\partial_{(k,\ell)}])A_{\mathrm{cl}}}{\partial_{(p,q)}}{Y_K} \nonumber\\
    \hspace{0.9cm}&- \tensor{\tproj\grad{f}_K}{\Gamma^{i,j}_{(k,\ell)(p,q)}(K) \partial_{i,j}}{Y_K}.
\end{align*}

\subsubsection{\acf{olqr}}\label{subsec:olqr}
The \ac{olqr} problem can be formulated as the optimization problem in \cref{eq:constained-lqr-opt} with the submanifold $\substableK = \stableK \cap \constraint_C$, where the constraint set $\constraint_C$ is defined as
\begin{gather*}
    \constraint_C \coloneqq \big\{K \in \Kmatrices \;|\; K = L C, \; L \in \Lmatrices \big\},
\end{gather*}
and $C \in \Cmatrices$ is the prescribed output matrix. For simplicity of presentation, we suppose that $C$ has full rank equal to $d\leq n$. Then, 
$\substableK = \stableK$ is a properly embedded submanifold of $\stableK$ with dimension $md$. Also, we can canonically identify each tangent space at $K \in \substableK$ with $T_K\substableK \cong \constraint_C$. Finally, 
at any $K \in \substableK$ and for any $E \in T_K \stableK$, the tangential projection of $E$ is $\tproj E = L^* C$ with $L^* \in \Lmatrices$ being the unique solution of the following linear equation
\begin{equation}\label{eq:tangent_olqr}
    L^* CY_K C^\top = E Y_K C^\top.
\end{equation}

\subsubsection{Complexity of \ac{ouralgo}}
\iftoggle{arXiv-version}{
The building block of our algorithm is solving \ac{dare}s; using the Bartels–Stewart algorithm for solving Sylvester equations, we have a complexity of $\mathcal{O}(n^3)$ for solving each \ac{dare}.
In the worst case, solving the Newton direction requires the Hessian coefficients, a tangential projection, Christoffel symbols, and the system of $|D|$-linear equations.
The complexity of solving a system of $n$-linear equations and a $n\times n$ matrix multiplication are both $\mathcal{O}(n^3)$ (without relying on more sophisticated algorithms). By symmetries of $h_{;(k,\ell)(p,q)}$, these require solving $|D|$ number of \ac{dare}s. Hence, both the tangential projection and the system of linear equations require $\mathcal{O}(|D|^3)$ operations, and the Hessian coefficients require $\mathcal{O}(n^3|D|)$ operations.}{}

The computational complexity of \ac{ouralgo} at each iteration is determined by the complexity of solving for the Newton direction $G$ as the stability certificate $\stabcer_K$ involves solving only one \ac{dare} within $\mathcal{O}(n^3)$ operations---say using the Bartels-Stewart algorithm. 
Thus, each iteration of QRNPO with $\overline{\mathrm{Hess}}$ has a rudimentary computational complexity of $\mathcal{O}(n^3|D| + |D|^3 + n^3) \approx \mathcal{O}(n^4m + n^3m^3)$ for the largest possible $|D| = n m$.  On the other hand,  QRNPO with $\mathrm{Hess}$ requires additional computations of approximately $|D|$ number of \ac{dare}s for the Christoffel symbols, but resulting in the same complexity. Finally, we note that efficient computation of the Christoffel symbols for these problems can benefit from additional structures such as sparsity; see \cite{talebi2022policy} for further discussion. 
% This that can be utilized to reduce the computational complexity--but this is not examined further in this work. 
% %
% %%%%%%%%%%%%%%%%%%%%%%%%%%%%%%%%%%%%%%%%%%%%%%%%%%%%%%%%%%%%%%%%%
\section{Numerical Results \label{sec:simulation}}
In this section, we provide numerical examples for optimizing the \ac{lqr} cost over submanifolds induced by \ac{slqr} and \ac{olqr} problems.
Recall that, for each of these problems, we can compute the coordinate functions of the covariant Hessian $h_{;(k,\ell)(p,q)}(K)$ with respect to the corresponding coordinate frame described in previous subsections. Therefore, finding the Newton direction $G$ at any point $K\in\substableK$ reduces to solving the system of linear equations for the unknowns $[G]^{k,\ell}$, as described in \S\ref{subsec:newton-direction}, and forming the Newton direction as $G = [G]^{k,\ell} \widetilde{\partial}_{(k,\ell)}|_K \in T_K\substableK$.

For each of \ac{slqr} and \ac{olqr} problems, we have simulated three different algorithms, the first two are the variants of \ac{ouralgo} where we use Riemannian connection or Euclidean connection to compute $\hess{h}$ or $\euchess{h}$, respectively. 
Note that although $\hess{h}_{K^*} = \euchess{h}_{K^*}$ whenever $\grad{h}_{K^*} = 0$ (as shown in \Cref{lem:nondegenerate}), 
this would not necessarily be the case where $\grad{h}$ does not vanish; therefore, we expect $\hess{h}$ and $\euchess{h}$ to contain distinct
information on neighborhoods of isolated local minima, that 
directly influence the performance of \ac{ouralgo} as will be discussed below.
The third algorithm is the \acf{pg} as studied in \cite{bu2019lqr}.
\ac{pg} is feasible for constraints in our examples as, under relevant assumptions, one is able to perform \ac{pg} updates by having access to merely the projection onto linear subspace of matrices--see for example~\cite[Theorem 7.1]{bu2019lqr}. Note that the Lipschitz estimate provided in \cite[Lemma 7.9]{bu2019lqr} is conservative.
Here, instead we choose a larger constant step size which generates stabilizing iterates and improves the performance of \ac{pg}--but not the convergence rate which remains sublinear.
 % However, as the convergence rate of the \ac{pg} is guaranteed to be sublinear, the improvement of error from optimality is not relatively significant.
Finally, for comparison, we also implement the \ac{npg} algorithm with the same constant stepsize where, comparing to \ac{pg}, the Euclidean projected gradient is replaced by the tangential projection of the Riemannian gradient.
Note that, both \ac{pg} and \ac{npg} algorithms have similar (initial) sublinear rates \cite{fatkhullin2020optimizing,bu2019lqr}; as such, despite their respective progress at the onset of the iterates, these algorithms--without further stepsize adjustment--become rather slow over time and not practically convergent. The code
for generating theses results can be found at  \cite{Talebi_Quasi_Riemannian_Newton_2022}.
% %%%%%%%%%%%%%%%%%%%%%%%%%%%%
\begin{example}[Trajectories of \ac{ouralgo} using $\hess$ versus $\euchess$]\label{exp:2}
In order to illustrate how the performance of \ac{ouralgo} is different in terms of using the Riemannian connection ($\hess{h}$) versus Euclidean connection ($\euchess{h}$), we consider an example with system parameters
\(\left(A|B|Q|R|\Sigma\right) = {\small \left(\begin{array}{cc|cc|cc|cc|cc}
                0.8 & 1.0 & 0.0 & 1.0 & 10.0 &  0.0 & 0.1 & 0.0 & 1.0 & 0.0\\
                0.0 & 0.9 & 1.0 & 0.0 & 0.0  &  0.5 & 0.0 & 0.1 & 0.0 & 5.0\end{array} \right)}.\)
% \begin{gather*}
%        \quad\quad\left(\begin{array}{cc|cc|cc|cc|cc}
%                 0.8 & 1.0 & 0.0 & 1.0 & 10.0 &  0.0 & 0.1 & 0.0 & 1.0 & 0.0\\
%                 0.0 & 0.9 & 1.0 & 0.0 & 0.0  &  0.5 & 0.0 & 0.1 & 0.0 & 5.0\end{array} \right).
% \end{gather*}   
We run \ac{ouralgo} and \ac{pg} algorithms for both \ac{slqr} and \ac{olqr} problems involving two decision variables, so that we can plot the trajectories of the iterates over the level curves of the associated cost functions from different initial conditions (as illustrated in
\Cref{fig:landscape-slqr} and \Cref{fig:landscape-olqr}, respectively). 
In this example, the stopping criterion for \ac{ouralgo} is when the error of iterates from optimality is below a small tolerance ($10^{-12}$); unless the Hessian fails to be positive definite for \ac{ouralgo} using $\euchess$, we run the algorithm for a fixed number of iterations. Since \ac{pg} does not practically converge even with large number of iterates due to its sublinear convergence rate, it is terminated when \ac{ouralgo}, using $\hess$, has converged.

In reference to \Cref{fig:landscape-slqr}, we first note that \ac{ouralgo} with $\euchess$ does not converge if initialized away from the local minimum (and away from the line $l_2 = -1$) as the Euclidean Hessian fails to be positive definite therein (see \Cref{fig:motivation}). %
On the other hand, \ac{ouralgo} with $\hess$ successfully captures the inherent geometry of the problem and converges from all initializations. 
These observations exemplify how \ac{ouralgo} can exploit the connection compatible with the metric (inherent to the cost function) in order to provide more effective iterate updates.
Second, the square marker on each trajectory of \ac{ouralgo} indicates the first time stepsize $\eta_t = 1$ is guaranteed to be stabilizing (\ie, $\stabcer_{K_t}\geq 1$). It can be seen that the neighborhood of the local minimum (zoomed in)--on which the identity stepsize is possible--is relatively small. Whereas, by using the stability certificate, the specific choice of stepsize adopted here enables \ac{ouralgo} to handle initialization further away from the local minimum.
\begin{figure}[t]
     \centering
     \includegraphics[trim=1.5 0 2cm 1.2cm, clip, width =0.5\textwidth]{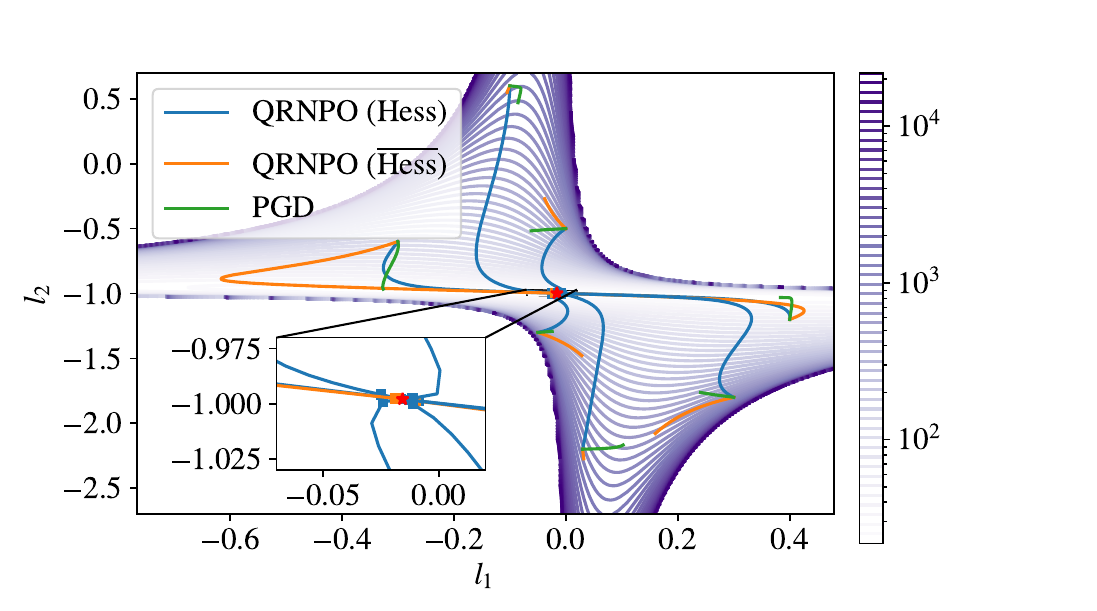}
     \caption{\small The trajectories of iterates 
     $K = \mathrm{diag}(l_1,l_2)$ generated by \ac{ouralgo} (with $\hess{h}$ and $\euchess{h}$) and \ac{pg}--from different initial points--for the \ac{slqr} problem with constraint $ D = \{(1,1),(2,2)\}$, over the level curves of $f$ in \cref{eq:constained-lqr-opt}.}
     \label{fig:landscape-slqr}
\vspace{-0.1cm}
\end{figure}%
\begin{figure}[t]
     \centering
     \includegraphics[trim=1.5 0 2cm 1.2cm, clip, width =0.5\textwidth]{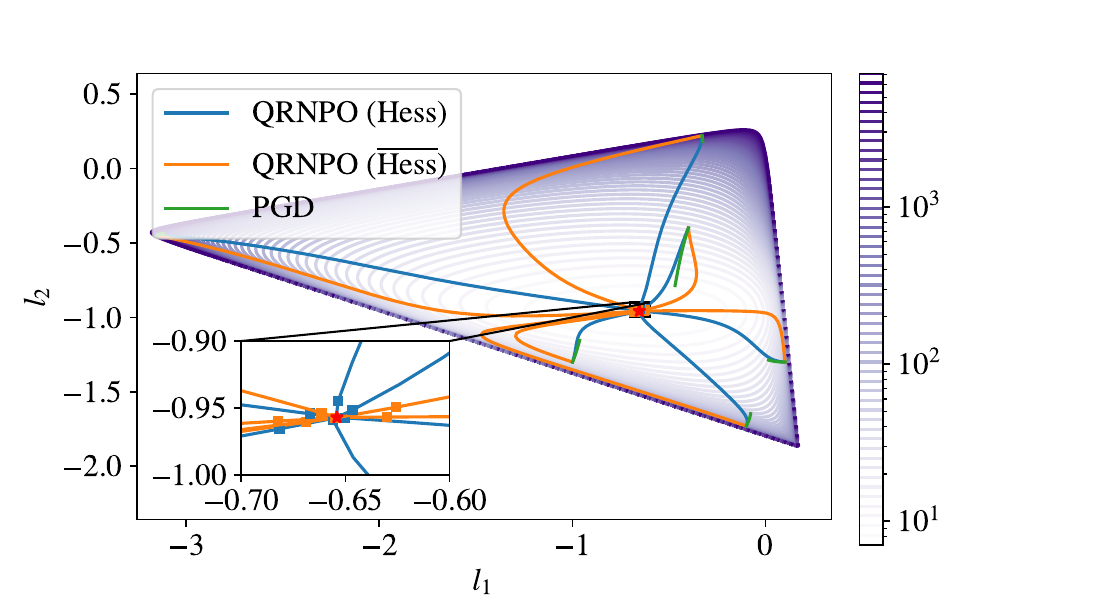}
     \caption{\small The trajectories of iterates $L = \left(\begin{array}{cc}l_1 & l_2 \end{array}\right)^\intercal$ generated by \ac{ouralgo} (with $\hess{h}$ and $\euchess{h}$) and \ac{pg}--from different initial points--for the \ac{olqr} problem with output matrix $ C = \left(\begin{array}{cc}1.0& 1.0 \end{array}\right)$, over the level curves of $f$ in \cref{eq:constained-lqr-opt}.}
     \label{fig:landscape-olqr}
\vspace{-0.3cm} \end{figure}
Next, notice that in both Figures~\ref{fig:landscape-slqr} and \ref{fig:landscape-olqr}, the trajectories of \ac{ouralgo} with $\hess$ are much more favorable in comparison to \ac{ouralgo} with $\euchess$, particularly, when initialized from points further away from the local minimum and closer to the boundary.
Additionally, similar to \Cref{fig:landscape-slqr}, the region on which the unit stepsize is guaranteed to be stabilizing is relatively small in \Cref{fig:landscape-olqr}.
\end{example}

% %%%%%%%%%%%%%%%%%%%%%%%%%%%%%%%%
\begin{example}[Randomly selected system parameters]\label{exp:3}
Next, we consider an example with $n=6$ number of states and  $m=3$ number of inputs, and simulate the behavior of \ac{ouralgo} and \ac{pg} for 100 randomly sampled system parameters. Particularly, the parameters $(A,B)$ are sampled from a zero-mean unit-variance normal distribution, where $A$ is scaled so that the open-loop system is stable, \ie, $K_0 = 0$ is stabilizing, and the pair is controllable. Furthermore, we choose $Q = \Sigma = I_n$ and $R = I_m$ in order to consistently compare the convergence behaviors across different samples. 
For the \ac{slqr} problem, we randomly sample for the sparsity pattern $D$ so that at least half of the entries are zero and all of them have converged from $K_0 = 0$ in less than 30 iterations.
For the \ac{olqr} problem, we also randomly sample the output matrix $C$ with $d =2$, where $98\%$ and $92\%$ of the corresponding iterates have converged from $K_0 = 0$ in less than 50 iterations using $\hess$ and $\euchess$, respectively.

The minimum, maximum and median progress of the three algorithms for both \ac{slqr} and \ac{olqr} problems are illustrated in \Cref{fig:random-slqr} and \Cref{fig:random-olqr}, respectively.
As guaranteed by \Cref{thm:convergence}, the linear-quadratic convergence behavior of \ac{ouralgo} is observed in these problems.
Especially, the quadratic behavior starts as soon as the stability certificate exceeds 1.
In both cases, \ac{ouralgo} with $\hess{h}$ (blue curves) built upon the Riemannian connection has a superior convergence rate compared with the case of using the Euclidean connection (orange curves); this was expected as the Riemannian connection is compatible with the metric induced by the geometry inherent to the cost function itself. This superior performance of \ac{ouralgo} with $\hess{h}$, in the meantime, requires computation of the Christoffel symbols.
\begin{figure}[t]
     \begin{subfigure}[b]{\columnwidth}
        \centering
         \includegraphics[trim=3 22 0 0, clip, width =\textwidth]{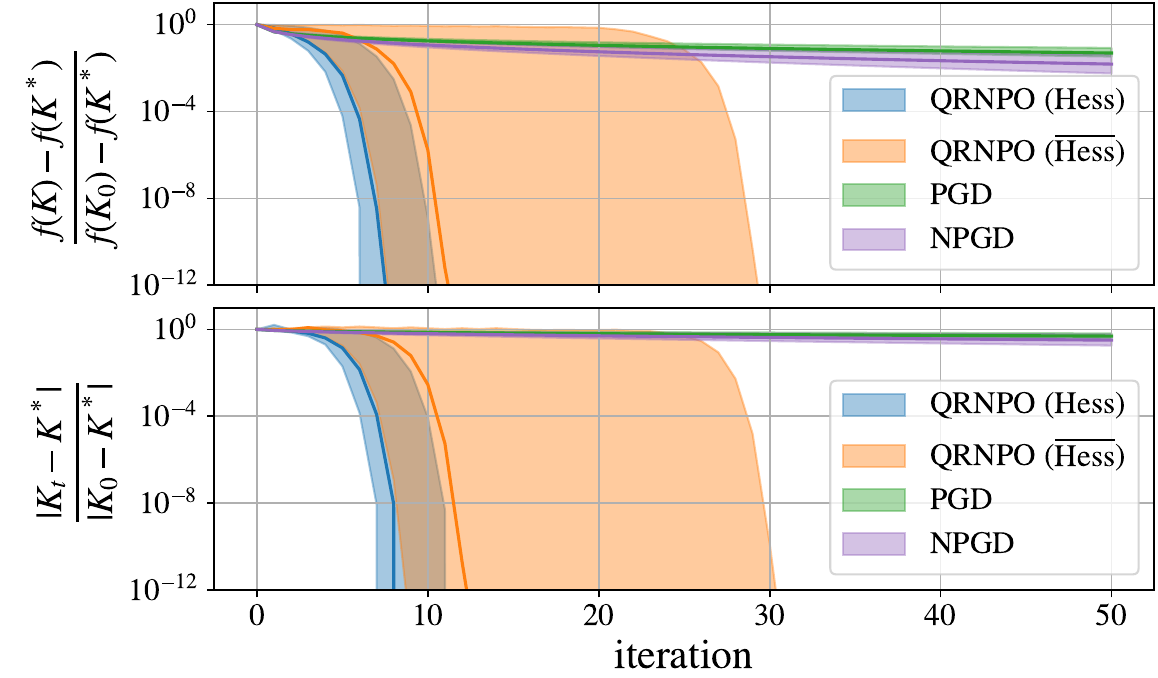}
         \caption{}
         \label{fig:random-slqr}
    \end{subfigure}
    \begin{subfigure}[b]{\columnwidth}
         \centering
         \includegraphics[trim=3 4 0 0, clip, width =\textwidth]{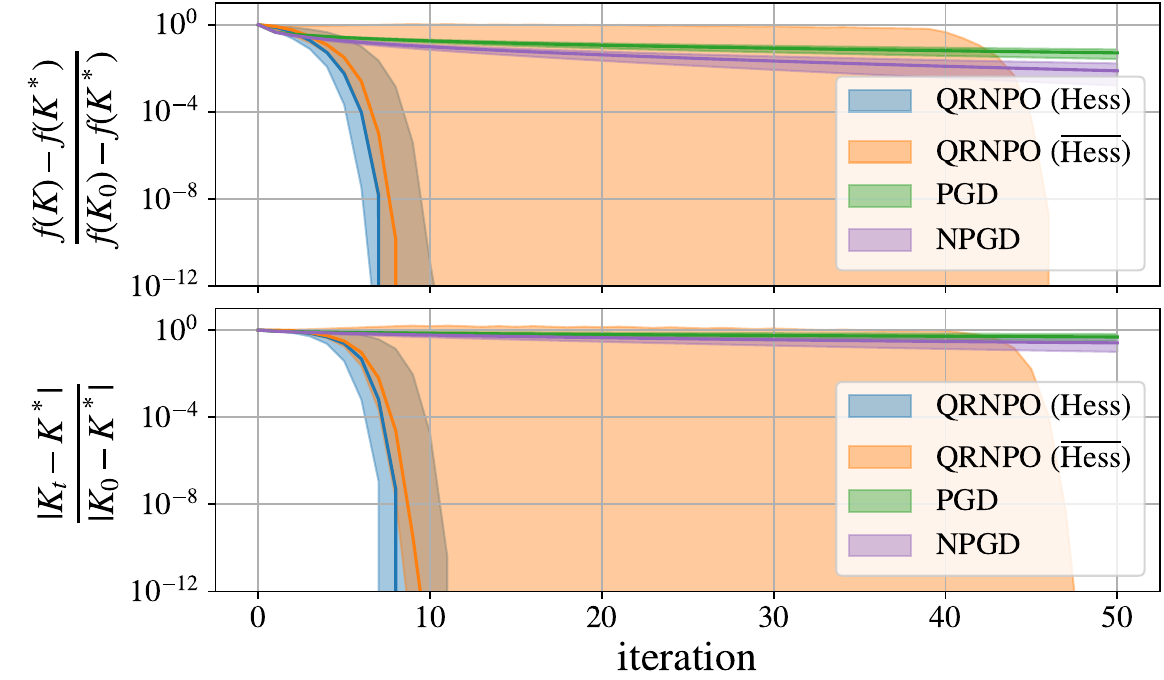}
         \caption{}
         \label{fig:random-olqr}
    \end{subfigure}
         \caption{\small The min, max and median progress of normalized error of iterates and cost values at each iteration of \ac{ouralgo} (with $\hess{h}$ and $\euchess{h}$), \ac{pg} and \ac{npg} for the (a) \ac{slqr} and (b) \ac{olqr} problems with 100 different randomly sampled system parameters, sparsity patterns and output matrices.}
\vspace{-0.3cm} 
\end{figure}
\end{example}

% %%%%%%%%%%%%%%%%%%%%%%%%%%%%%%%%%%%%%%%%%%%%%%%%%%%%%%%%%%%%%%%%%
\section{Conclusions and Future Directions}
\label{sec:conclusion}
In this work, we considered the problem of optimizing a smooth function over submanifolds of Schur stabilizing controllers $\stableK$. In order to treat this problem in a more general setting, we studied the first and second order behavior of a smooth function when constrained to an embedded submanifold from an extrinsic point of view.
Subsequently, using the second order information of the restricted function, we developed an algorithm that guarantees convergence to local minima--at least with a linear rate--and eventually with a quadratic rate.
Combining this approach with backtracking line-search techniques or positive definite modifications of the Hessian operator \cite{dennis1996numerical, nocedal2006numerical} can be considered as immediate future directions for a global convergence analysis.

Even though the proposed algorithm depends on the linear structure of $\substableK$, the machinery developed here can be utilized for other
submanifolds, a topic that will be considered in our future works.
For example, in contrast to the \ac{slqr} and \ac{olqr} problems considered, we can explore how a constraint on the \emph{average input energy} translates to a \textit{nonlinear} constraint that pertains to the inherent geometry of the \ac{lqr} problem.
In this direction, we define the \textit{average input energy},
\(\textstyle E_\uu \coloneqq \mathop{\mathbb{E}}_{\x_0 \sim \mathcal{D}} \|\uu\|_{\ell_2}^2,\)
where $\|.\|_{\ell_2}$ refers to $\ell_2$-norm. 
If a static linear policy, \ie, $\uu = K \x$ for $K\in \stableK$, is desired for this problem setup,
then the closed-loop system assumes the form
$\x_k = (A_{\mathrm{cl}})^k \x_0$; one can now
show that under the usual identification of $T_K\stableK$ with $\Kmatrices$, we have $E_\uu(K) = \rnorm{K}{K}^2$. But then, $E_\uu\colon\stableK \mapsto \bR$ is smooth by composition, and as such, every regular level set of $E_\uu$ translates to an upperbound on the average input energy. Regular Level Set Theorem now implies that the average input energy optimal control synthesis can be pursued via an embedded submanifold of $\stableK$ which has a nonlinear but simple structure, whenever considered in the associated Riemannian geometry. Solving this problem still requires an efficient retraction that would substitute the linear updates possible in \ac{slqr} and \ac{olqr} problems. The framework discussed in this work allows the integration of such retractions in the synthesis procedure. As such, the proposed work
opens up a new approach for solving a wide range of constrained  optimization problems over the manifold of Schur stabilizing controllers.

% %%%%%%%%%%%%%%%%%%%%%%%%%%%%%%%%%%%%%%%%%%%%%%%%%%%%%%%%%%%%%%%%%
\section*{Acknowledgments}
The first author thank Professor John M. Lee for his inspiring lectures on differential geometry and insightful comments on this manuscript. The authors also thank Jingjing Bu for helpful discussions on first order methods for control, as well as the Associate Editor and anonymous reviewers for constructive feedback and suggestions that have been reflected in this manuscript. 

% %%%%%%%%%%%%%%%%%%%%%%%%%%%%%%%%%%%%%%%%%%%%%%%%%%%%%%%%%%%%%%%%%
\appendix
\label{appendix}
\begin{proof}[On the Hessian Operator]\label{app:diffgeom} The Hessian operator (denoted by $\hess{f}$) as introduced in \S\ref{sec:extrinsic} is well-defined and the value of $\hess{f}[U]$ at any $K \in\stableK$ depends only on $U_K$; this is due to the property for the connection.
Note that
\begin{multline}
    \tensor{\hess{f}[U]}{W}{Y} 
    = U\tensor{\grad{f}}{W}{Y} - \tensor{\grad{f}}{\nabla_U W}{Y} \\
    = U(W f) - (\nabla_U W)f = W(U f) - (\nabla_W U)f \label{eq:hess-compute},
\end{multline}
$\forall U, W \in \fX(\stableK)$, where the first equality is the consequence
of having the Riemannian connection compatible with the metric, the second one is by 
the definition of $\grad{f}$, and the last one is due to symmetry of the Riemannian connection. Thus, by \cref{eq:hess-compute}, the Hessian operator is self-adjoint, \ie,
\[\tensor{\hess{f}[U]}{W}{Y} = \tensor{U}{\hess{f}[W]}{Y}.\]
Similarly, we can consider $\hess{h}$ for any smooth function $h \in C^{\infty}(\substableK)$, where we consider the submanifold $\substableK \subset \stableK$ with the induced Riemannian metric and the associated Riemannian connection of $\stableK$.

Next, for the computational purposes, we would like to introduce the \emph{covariant Hessian} of $f$ with respect to $\metric$, denoted by $\nabla^2 f$ \cite[Proposition 4.17]{lee2018introduction}. It is a 2-tensor field obtained by taking total covariant derivative of $f$ twice. The Riemannian connection is symmetric, and so is the covariant Hessian. Furthermore, the covariant Hessian and Hessian operator are related as,
\begin{gather}\label{eq:covariantHessian}
\begin{small}
\begin{aligned}
\nabla^2 f [W,U] = \nabla_U \nabla_W f - (\nabla_U W) f =
\tensor{\hess{f}[U]}{W}{Y},
\end{aligned}
\end{small}
\end{gather}
where the last equality follows by \cref{eq:hess-compute}.

Recall now that $\grad{f} = (\diff f)^\sharp$, where $\sharp$ denotes the index raising operator, referred to as the \emph{sharp} operator \cite[Chapter 2]{lee2018introduction}. Also, $\hess{f}\colon\mathfrak{X}(\stableK) \mapsto \mathfrak{X}(\stableK)$ can be viewed as the total covariant derivative of $\grad{f}$, \ie, $\hess{f} = \nabla \grad{f}$. Then
\begin{equation}\label{eq:hess-vs-covariant-Hessian}
    \hess{f} = \nabla (\diff f)^\sharp = \left(\nabla (\diff f) \right)^\sharp = \left(\nabla^2 f \right)^\sharp,
\end{equation}
where the equality in the middle follows by the fact the index raising operator commute with the covariant derivative operator, and the last equality is due to the definition of connection for a smooth function $f \in C^\infty(\stableK)$. Note that in \cref{eq:hess-vs-covariant-Hessian}, the index raising refer to the second argument of $\nabla^2 f$. However, as the covariant Hessian of any smooth function is a symmetric 2-tensor field, the index raising could be with respect to any of the entries. 
Finally, similar definitions and relations as discussed above are available for $h$ as a smooth function on the embedded Riemannian submanifold $\substableK$ with the induced metric and corresponding connection; these are omitted for brevity. 
\end{proof}

% %%%%%%%%%%%%%%%%%%%%%%%%%%%%%%%%%
\begin{proof}[Proof of \Cref{lem:dlyap}]
Since $\rho\colon\Amatrices \mapsto \bR$ is a continuous map, $\Astable$ is an open subset of $\Amatrices$ and thus an open submanifold. For each $A \in \Astable$, by Lyapunov Stability Criterion, there exists a unique solution $X$ to \cref{eq:lyap-gen} which has the infinite-sum representation.
But, as for each $A \in \Astable$, the series converges, each matrix entry of $X$ can be written as a convergent power series of elements of $A$ and $Z$. Therefore, each matrix entry of $X$ is a real analytic function of several variables (as defined in \cite{krantz2002primer}) on the open subset $\Astable \times \Amatrices \subset \Amatrices \times \Amatrices$. Hence, we conclude that $\lyap$ is a well-defined smooth map.\footnote{An alternative argument can be provided by the closed form solution of \cref{eq:lyap-gen} and its \emph{vectorization} involving rational functions of several variables with non-vanishing denominators--cf. Lemma 3.6 in \cite{bu2019lqr}.}
Next, under the identification in the premise, 
it follows that,
\begin{gather*}
    \diff \lyap_{(A,Q)}[E,F] = \diff \lyap_{(A,Q)}[E,0] +  \diff \lyap_{(A,Q)}[0,F].
\end{gather*}
However, $\lyap$ is linear in the second entry, so $\diff \lyap_{(A,Q)}[0,F] = \lyap(A,F)$. Also since $\Astable$ is open, for small enough $\varepsilon$, $\gamma\colon [0,\varepsilon] \mapsto \Astable \times \Amatrices$ with $\gamma(t) = (A + t E, Q)$ is a well-defined smooth curve starting at $(A,Q)$ whose initial velocity is $(E,0)$. Then,
\[\diff \lyap_{(A,Q)}[E,0] = {d}/{dt}\big|_{t=0} \lyap \circ \gamma(t).\]
Let $X_t \coloneqq \lyap \circ \gamma(t)$ and $X \coloneqq \lyap \circ \gamma(0)$; then we obtain,
\begin{align*}
    X_t - X 
    &= \lyap \left(A, t(E X A^\intercal + A X E^\intercal) + \mathcal{O}(t^2) \right)\\
    &= t \lyap(A, E X A^\intercal + A X E^\intercal) + \mathcal{O}(t^2),
\end{align*}
where the first equality is by direct algebraic manipulation and the second one follows by linearity of $\lyap$ in the second entry.
Therefore, $\diff \lyap_{(A,Q)}[E,0] = \lyap(A, E X A^\intercal + A X E^\intercal)$, and the first claim follows by adding the two computed differentials and using linearity of $\lyap$ in the second entry again. Finally, note that any square matrix has a spectrum identical to its transpose; therefore if $A \in \Astable$ then $A^\intercal \in \Astable$, and thus the last property follows by the convergent series representations of $\lyap(A^\intercal, Q)$ and $\lyap(A, \Sigma)$, as well as the cyclic permutation property of trace.
\end{proof}

% %%%%%%%%%%%%%%%%%%%%%%%%%%%%%%%%%
\begin{proof}[Proof of \Cref{lem:tensor}]
By \Cref{lem:dlyap} for each $K\in \stableK$, $\lyap(A_{\mathrm{cl}},\Sigma_K)$ is uniquely determined, symmetric and smooth in $K$, since $A_{\mathrm{cl}}$ is stabilizing. Also, $\lyap(A_{\mathrm{cl}},\Sigma_K) \in \Amatrices$ is positive semidefinite by observing the infinite-sum representation of solution to Lyapunov equation and the fact that $\Sigma_K \succeq 0$.
Next, $\tr{(V_K)^\intercal W_K \lyap(A_{\mathrm{cl}},\Sigma_K)}$ is a smooth function of elements of $V_K, W_K$ and $\lyap(A_{\mathrm{cl}},\Sigma_K)$.
Therefore, for any $V,W \in \fX(\stableK)$, the function $\tensor{V}{W}{Y}$, as defined in the premise, is well-defined and smooth on $\stableK$.
Additionally, by linearity of trace, we observe that $\tensor{\cdot}{\cdot}{Y}$ is multilinear over $C^\infty(\stableK)$, \ie,
\[\tensor{f U + h V}{W}{Y} = f\tensor{U}{W}{Y} + h\tensor{V}{W}{Y},\]
for any $f,h \in C^\infty(\stableK)$ and $U,V \in \fX(\stableK)$, and similarly for the second entry. Therefore, by Tensor Characterization Lemma \cite[Lemma 12.24]{lee2013smooth}, it is induced by a smooth covariant 2-tensor field. Finally, symmetry follows from the
fact that for any $K \in \stableK$, the cyclic property of trace implies,
\begin{multline*}
    \tensor{V}{W}{Y}\big|_K = \tr{(\lyap(A_{\mathrm{cl}},\Sigma_K))^\intercal \; (W_K)^\intercal \; V_K}\\
    =\tr{(W_K)^\intercal \; V_K \; (\lyap(A_{\mathrm{cl}},\Sigma_K))^\intercal} = \tensor{W}{V}{Y}\big|_K,
\end{multline*}
as $\lyap(A_{\mathrm{cl}},\Sigma_K)$ is symmetric.
\end{proof}

% %%%%%%%%%%%%%%%%%%%%%%%%%%%%%%%%%
\begin{proof}[Proof of \Cref{prop:Riemannian-metric}]
We know that $\stableK$ is a smooth manifold, and by \Cref{lem:tensor} and Smoothness Criteria for Tensor Fields \cite[Proposition 12.19]{lee2013smooth}, $\metric$ is a smooth symmetric covariant 2-tensor field. Hence, it suffices
to show that it is positive definite at each point $K \in \stableK$. But $\Sigma_K \succeq 0$ and $(A_{\mathrm{cl}},\Sigma_K)$ is controllable; therefore $Y_K = \lyap(A_{\mathrm{cl}},\Sigma_K)$ is a positive definite matrix, implying that $\metric_K(E,E) = \tr{(E Y_K^{\sfrac{1}{2}})^\intercal E Y_K^{\sfrac{1}{2}}} \geq 0$, for any $E \in T_K \stableK$ with equality if and only if $E$ is the zero element. Next, to compute the coordinate representation of $\metric$, for each coordinate pairs $(i,j)$ and $(k,\ell)$ we have,
\begin{gather*}
    \metric_{(i,j) (k,\ell)}(K) = \metric_K(\partial_{i,j}|_K,\partial_{k,\ell}|_K)
    = \tr{(\partial_{i,j}|_K)^\intercal \partial_{k,\ell}|_K Y_K},
\end{gather*}
under the usual identification of $T_K \stableK$. Since under this identification $\partial_{i,j}$ corresponds to the element of $\Kmatrices$ with entry $1$ in $(i,j)$-th coordinate and zero elsewhere, the expression for $\metric_{(i,j) (k,\ell)}$ follows by direct computation of the last equality.
Finally, by definition of \emph{inverse matrix}, for each $(i,j)$ and $(k,\ell)$, we must have $\sum_{r,s} \metric^{(i,j)(r,s)} \metric_{(r,s)(k,\ell)} = 1$ if $(i,j)=(k, \ell)$ and $0$ otherwise.
Next, for each $i = k$, let $\left[\metric^{(k,.)(k,.)}\right]$ denote the matrix with $\metric^{(k,j)(k,s)}$ as its $(j,s)$th entry. Then, by the expression for $\metric_{(i,j) (k,\ell)}$, it must satisfy
\(\left[\metric^{(k,.)(k,.)}\right] Y_K = I_n,\)
and therefore $[\metric^{(k,.)(k,.)}] = Y_K^{-1}$ as $Y_K \succ 0$. The rest of the reasoning follows by performing a similar computation for each $i\neq k$ and noting the zero pattern in the expression for $\metric_{(i,j) (k,\ell)}$.
\end{proof}

% %%%%%%%%%%%%%%%%%%%%%%%%%%%%%%%%%
\begin{proof}[Proof of \Cref{lem:christoffel}]
We know that
\begin{multline}\label{eq:christoffel-gen}
    \textstyle\Gamma^{(i,j)}_{(k,\ell)(p,q)} = \sum_{r,s} ({{\metric^{(i,j)(r,s)}}/{2}})  \big( \partial_{(k,\ell)} \metric_{(r,s)(p,q)} \\
    \textstyle + \partial_{(p,q)} \metric_{(r,s)(k,\ell)} - \partial_{(r,s)} \metric_{(k,\ell)(p,q)} \big),
\end{multline}
for any $(i,j),(k,\ell),(p,q) \in [m]\times [n]$,
where $\metric^{(i,j)(r,s)}$ denotes the inverse matrix of $\metric_{(i,j)(r,s)}$.
If $k \neq i \neq p \neq k$, then by sparsity pattern in the expression for $\metric_{(i,j) (k,\ell)}$ in \Cref{prop:Riemannian-metric}, we obtain $\Gamma^{(i,j)}_{(k,\ell)(p,q)} = 0$.
 Next, if $k = i \neq p$, then \cref{eq:christoffel-gen} simplifies to
\begin{gather*}
    \sum_{s} \frac{\metric^{(i,j)(i,s)}}{2}  \Big( \partial_{(p,q)} \metric_{(i,s)(i,\ell)} \Big) = \frac{1}{2}\left( \partial_{(p,q)} [Y_K]_{(\ell,.)} \right)[Y_K^{-1}]_{(.,j)}.
\end{gather*}
Next, for any fix $i,p$ and $q$, let $\Gamma^{(i,.)}_{(i,.)(p,q)}$ denote the $n\times n$ matrix with $\Gamma^{(i,j)}_{(i,\ell)(p,q)}$ as its $(\ell,j)$ entry. Then it must satisfy
\(\Gamma^{(i,.)}_{(i,.)(p,q)} = \sfrac{1}{2} \left( \partial_{(p,q)} Y_K \right) Y_K^{-1},\)
where $\partial_{(p,q)} Y_K$ indicates the action of tangent vector $\partial_{(p,q)}$ on the composite map $K \mapsto (A_{\mathrm{cl}},\Sigma_K) \xrightarrow{\lyap} Y_K$.
By \Cref{lem:dlyap}, we can compute,
\begin{gather*}
\begin{aligned}
    \partial_{(p,q)} Y_K = \diff \lyap_{(A_{\mathrm{cl}},\Sigma_K)} \left[B \partial_{(p,q)}, \; \partial_{(p,q)}^\intercal \Sigma_2 K + K^\intercal \Sigma_2 \partial_{(p,q)}\right]%\\
    %&= \lyapij_K{(p,q)},
\end{aligned}
\end{gather*}
under the usual identification of $T_K \stableK$. Thus, $\partial_{(p,q)} Y_K = \lyapij_K{(p,q)}$ which proves the second case. The third case follows by the symmetry of the Riemannian connection, \ie, $\Gamma^{(i,j)}_{(k,\ell)(p,q)} = \Gamma^{(i,j)}_{(p,q)(k,\ell)}$.
Next, if $p = k \neq i$, then by \Cref{prop:Riemannian-metric}, \cref{eq:christoffel-gen} simplifies to
\begin{gather*}
    \sum_{s} \frac{\metric^{(i,j)(i,s)}}{2}  \big( \partial_{(i,s)} \metric_{(k,\ell)(k,q)} \big) = \frac{1}{2} \sum_{s} \left[ \partial_{(i,s)} Y_K \right]_{(q,\ell)} [Y_K^{-1}]_{(s,j)} ,
\end{gather*}
with $\partial_{(i,s)} Y_K = \lyapij_K{(i,s)}$ computed similarly.
Finally, if $k = i = p$ then similarly \cref{eq:christoffel-gen} simplifies to
\begin{gather*}
    \sum_{s} \frac{\metric^{(i,j)(i,s)}}{2}  \Big( \partial_{(i,\ell)} \metric_{(i,s)(i,q)} + \partial_{(i,q)} \metric_{(i,s)(i,\ell)} - \partial_{(i,s)} \metric_{(i,\ell)(i,q)} \Big),
\end{gather*}
and substituting each term similarly completes the proof.
\end{proof}

% %%%%%%%%%%%%%%%%%%%%%%%%%%%%%%%%%
\begin{proof}[Proof of \Cref{prop:extrinsic}]
Note that $f$ is smooth and $\substableK$ is an embedded submanifold of $\stableK$. Therefore, $h\colon\substableK \mapsto \bR$ is smooth by restriction and we can define $\grad{h}$ and $\hess{h}$ on $\substableK$.
But, $\grad{h} \in \fX(\substableK)$ is the unique vector field on $\substableK$ such that $\submetric(W,\grad{h}) = W h$ for any $W \in \fX(\substableK)$. 
Unraveling the definition implies that for any $K \in \substableK\subset \stableK$,
\begin{align*}
    \diff h_K(W_K) &= \diff f_K ( \diff \iota_{\footnotesize\substableK} (W_K)) =
    \metric_K\left(  \diff \iota_{\footnotesize\substableK}(W_K), \grad{f}_K\right) \\
    &= \metric_K\left(  \diff \iota_{\footnotesize\substableK}(W_K), \diff \iota_{\footnotesize\substableK} (\tproj \grad{f}_K)\right)
    ,
\end{align*}
as $h = f \circ \iota_{\footnotesize\substableK}$ and thus $\diff h = \diff f \circ \diff \iota_{\footnotesize\substableK}$, where the last equality follows by the fact that $\iota_{\footnotesize\substableK}(W_K) \in T_K\stableK$ is tangent to $\substableK$. By definition of tangential projection, $\tproj (\grad{f}|_{\footnotesize\substableK})$ is then a vector field on $\substableK$ that satisfies
\[W h = \submetric\left(W, \tproj \grad{f} |_{\footnotesize\substableK}\right),\]
for any $W \in \fX(\substableK)$.
Therefore, the first claim follows by uniqueness of the gradient.
Next, note that the Hessian operator of $h \in C^\infty(\substableK)$ is defined as
\(\hess{h}[V] \coloneqq \widetilde{\nabla}_V \grad{h},\)
for any $V \in \fX(\substableK)$. But then, the first claim together with  \cref{eq:subconnection} and the linearity of connection imply that
\begin{align*}
    \hess{h}[V] 
    =& \tproj \nabla_V (\tproj (\grad{f}|_{\footnotesize\substableK})) \\
    =&\tproj (\hess{f}[V]\big|_{\footnotesize\substableK}) - \tproj \nabla_V (\nproj (\grad{f}|_{\footnotesize\substableK})),
\end{align*}
where all $V$, $\tproj (\grad{f}|_{\footnotesize\substableK})$ and $\nproj (\grad{f}|_{\footnotesize\substableK})$ are extended arbitrarily to vector fields on a neighborhood of $\substableK$ in $\stableK$. Finally, the extrinsic expression of $\hess{h}$ follows by The Weingarten Equation \cite[Proposition 8.4]{lee2018introduction}, indicating that $\tproj \nabla_V (\nproj (\grad{f}|_{\footnotesize\substableK})) = - \wein_{\nproj (\grad{f}|_{\footnotesize\substableK})}[V]$.
\end{proof}

% %%%%%%%%%%%%%%%%%%%%%%%%%%%%%%%%%
\begin{proof}[Proof of \Cref{lem:nondegenerate}]
Let $\widetilde{\gamma}\colon (-\varepsilon, \varepsilon) \mapsto \substableK$ denote the smooth geodesic curve on the submanifold $\substableK$ with $\widetilde{\gamma}(0) =K^*$ and $\widetilde{\gamma}'(0) =F$ for an arbitrary $F \in T_{K^*}\substableK$. Define $\ell(t) \coloneqq h \circ \widetilde{\gamma}(t)\colon (-\varepsilon, \varepsilon) \mapsto \bR$ which is smooth by composition. Then, $K^*$ is a local minimum for $h$, so is $t=0$ for $\ell(t)$ following by smoothness of $\widetilde{\gamma}$. Therefore,
\begin{gather*}
    0 = \ell'(0) = \diff h_{\widetilde{\gamma}(0)} \circ \widetilde{\gamma}'(0) = \tensor{\grad{h}_{K^*}}{F}{Y_{K^*}},
\end{gather*}
and as $F$ was an arbitrary tangent vector, we conclude that $\grad{h}_{K^*} = 0$. Now, recall that non-degenerate critical points are isolated \cite[Corollary 2.3]{milnor2016morse}.
Next, Taylor's formula for $\ell$ at $t= 0$ yields
\[\ell(t) = \ell(0) + t\tensor{\grad{h}_{K^*}}{F}{Y_{K^*}} + \sfrac{1}{2} \ell''(s) t^2,\]
for some $s \in(0, t)$. 
As $\grad{h}_{K^*} = 0$ and $t=0$ is a local minimum of $\ell(t)$, we must have $\ell''(s) \geq 0$; by tending $t \mapsto 0$, smoothness of $\ell$ implies that $\ell''(0) \geq 0$. 
But,
\begin{gather*}
    \ell''(t) = \widetilde{D}_t \tensor{\grad{h}_{\widetilde{\gamma}(t)}}{\widetilde{\gamma}'(t)}{Y_{\widetilde{\gamma}(t)}} = \tensor{\widetilde{D}_t\grad{h}_{\widetilde{\gamma}(t)}}{\widetilde{\gamma}'(t)}{Y_{\widetilde{\gamma}(t)}}
\end{gather*}
where $\widetilde{D}_t$ denotes the covariant derivative along $\widetilde{\gamma}$ on $\substableK$, and the last equality follows by its compatibility with the metric and the fact that $\widetilde{\gamma}$ is a geodesic (so that $\widetilde{D}_t \widetilde{\gamma}'(t) \equiv 0$). As $\grad{h}|_{\widetilde{\gamma}(t)} \in \fX(\gamma)$ is clearly extendable, we conclude that
\begin{gather*}
    \ell''(t) = \tensor{\nabla_{\widetilde{\gamma}'(t)}\grad{h}\big|_{\widetilde{\gamma}(t)}}{\widetilde{\gamma}'(t)}{Y_{\widetilde{\gamma}(t)}} = \tensor{\hess{h}_{\widetilde{\gamma}(t)}[\widetilde{\gamma}'(t)]}{\widetilde{\gamma}'(t)}{Y_{\widetilde{\gamma}(t)}}
\end{gather*}
and thus particularly $\ell''(0) = \tensor{\hess{h}_{K^*}[F]}{F}{Y_{K^*}}$.
Since $F$ was arbitrary and $K^*$ is nondegenerate, $\ell''(0) \geq 0$ implies that $\hess{h}_{K^*}$ is positive definite. Next, existence of a neighborhood at $K^*$ on which $\hess{h}$ is positive definite follows by smoothness--in particular continuity--of the operator $\hess{h}_K$ in $K$.
Finally, let $\widetilde{\nabla}$ and $\widetilde{\Enabla}$ denote the connections on $T\substableK$ induced, respectively, by the connections $\nabla$ and $\Enabla$ on $T\stableK$. Then by The Difference Tensor Lemma, the difference tensor between $\widetilde{\nabla}$ and $\widetilde{\Enabla}$--defined as $D(U,V)\coloneqq \widetilde{\nabla}_U V - \widetilde{\Enabla}_U V$ for any $U, V \in \fX(\substableK)$--is indeed a (1,2)-tensor field. That means, as $\grad{h}_{K^*} = 0$,
\begin{gather*}
\hess{h}_{K^*}[U_{K^*}] - \euchess{h}_{K^*}[U_{K^*}]
= D(U, \grad{h})|_{K^*}
% = D(U_{K^*}, 0) 
= 0.
\end{gather*}%
The last claim then follows as $U \in \fX(\substableK)$ was arbitrary. 
\end{proof}

% %%%%%%%%%%%%%%%%%%%%%%%%%%%%%%%%%
\begin{proof}[Proof of \Cref{thm:convergence}]
By \Cref{lem:nondegenerate}, $\grad{h}_{K^*} = 0$ and there exists a neighborhood $\nghbd$ of $K^*$ on which $\hess{h}_{K}$ is positive. Furthermore, by continuity of $\hess{h}$ (and, if necessary, shrinking $\nghbd$) we can obtain constant positive scalars $m$ and $M$ such that for all $K \in \nghbd$ and $G \in T_K\substableK$, 
\begin{equation}\label{eq:mMbounds}
    m \rnorm{G}{K}^2 \leq \tensor{\hess{h}_K[G]}{G}{Y_K} \leq M \rnorm{G}{K}^2,
\end{equation}
where $\rnorm{\cdot}{K}$ denotes the norm induced by $\metric$ at $K$.
In particular, if $G_t \in T_{K_t}\substableK$ is the Newton direction at some point $K_t \in \nghbd$, then (by Cauchy-Schwartz inequality at $K_t$)
\begin{equation}\label{eq:bound-nd}
    \rnorm{G_t}{K_t} \leq ({\large\sfrac{1}{m}})
    \rnorm{\grad{h}_{K_t}}{K_t}.
\end{equation}
Next, define the curve $\gamma\colon [0,\stabcer_{K_t}] \mapsto \substableK$ with $\gamma(\eta) = K_t + \eta G_t$, and consider a smooth parallel vector field (with respect to the Riemannian connection) $E(\eta)$ along $\gamma$--refer to \cite{lee2018introduction} for \emph{parallel vector fields along curves} and \emph{parallel transport}. Also, define $\phi\colon [0,\stabcer_{K_t}] \mapsto \bR$ with 
\(\phi(\eta) \coloneqq \tensor{\grad{h}_{\gamma(\eta)}}{E(\eta)}{Y_{\gamma(\eta)}}.\)
Notice that $\grad{h}$ is smooth, so is $\phi$ and by compatibility with the metric and that $\grad{h}_{\gamma(\eta)}$ is clearly extendable, we have
\begin{gather*}
    \phi'(\eta) = \tensor{D_{\eta}\grad{h}_{\gamma(\eta)}}{E(\eta)}{Y_{\gamma(\eta)}} = \tensor{\hess{h}_{\gamma(\eta)}[G_t]}{E(\eta))}{Y_{\gamma(\eta)}},
\end{gather*}
where $D_{\eta}$ is the covariant derivative along $\gamma$ and $G_t$ is extended to the vector field along $\gamma$ with constant coordinates in the global coordinate frame. Thus, as
\[\textstyle \phi(\eta) = \phi(0) + \eta \phi'(0) + \int_0^\eta [\phi'(\tau)-\phi'(0)] d\tau,\]
by direct substitution and the fact that $G_t$ is the Newton direction at iteration $t$, we obtain that
\begin{multline*}
    \phi(\eta) = (\eta -1)\tensor{\hess{h}_{K_t}[G_t]}{E(0)}{Y_{K_t}}\\
    +\textstyle \int_0^{\eta} \tensor{[\hess{h}_{\gamma(\tau)} - \ptrans^\gamma_{0,\tau} \hess{h}_{\gamma(0)}]G_t}{E(\tau))}{Y_{\gamma(\tau)}} d\tau,
\end{multline*}
where $\ptrans^\gamma_{0,\tau}$ denotes the parallel transport from $0$ to $\tau$ along $\gamma$.
Again, as every parallel transport map along $\gamma$ is a linear isometry we claim that
\begin{gather*}
    \tensor{\grad{h}_{K_{t+1}}}{E(\eta_t)}{Y_{K_{t+1}}} = (\eta_t -1)\tensor{\ptrans^\gamma_{0,\eta_t}\hess{h}_{K_t}[G_t]}{E(\eta_t)}{Y_{K_{t+1}}}\\
    \textstyle \quad + \int_0^{\eta_t} \tensor{\ptrans^\gamma_{\tau,\eta_t}[\hess{h}_{\gamma(\tau)} - \ptrans^\gamma_{0,\tau} \hess{h}_{\gamma(0)}]G_t}{E(\eta_t))}{Y_{K_{t+1}}} d\tau.
\end{gather*}
Note that, for each $\tau \in [0,\stabcer_{K_t}]$, $\hess{h}_{\gamma(\tau)}$ is a self-adjoint operator that is smooth in $\tau$ as $\gamma$ is. So, by \cref{eq:mMbounds}, we obtain
\[ \rnorm{\ptrans^\gamma_{0,\eta_t}\hess{h}_{K_t}[G_t]}{K_{t+1}}
\leq M \rnorm{G_t}{K_{t}}\]
and by smoothness there exist a constant $L>0$ such that
\begin{equation*}
    \rnorm{\ptrans^\gamma_{\tau,\eta_t}[\hess{h}_{\gamma(\tau)} - \ptrans^\gamma_{0,\tau} \hess{h}_{\gamma(0)}]G_t}{K_{t+1}}
    \leq \tau L \rnorm{G_t}{K_{t}}^2 
\end{equation*}
where we used the isometry of parallel transport again in obtaining the bounds. Therefore, by choosing the parallel vector field $E(\eta)$ along $\gamma$ such that $E(\eta_t) = \grad{h}_{K_{t+1}}$ we obtain that
\begin{multline}\label{eq:grad-iter-bound}
    \rnorm{\grad{h}_{K_{t+1}}}{K_{t+1}}\hspace{-0.2cm} \leq M|1-\eta_t| \rnorm{G_t}{K_{t}} + ({\large\sfrac{\eta_t^2 L}{2}}) \rnorm{G_t}{K_{t}}^2\\
    \leq \frac{M|1-\eta_t|}{m} \rnorm{\grad{h}_{K_t}}{K_t} \hspace{-0.2cm} 
    + \frac{L\eta_t}{2m^2} \rnorm{\grad{h}_{K_t}}{K_t}^2
\end{multline}
where the last inequality follows by \cref{eq:bound-nd} and since $\eta_t \leq 1$.
Next, let $F_{t+1} \in T_{K_{t+1}}\substableK$ be tangent vector that $\xi(\eta) = \widetilde{\Exp}_{K_{t+1}}[\eta F_{t+1}]$ is the minimum-length geodesic in $\substableK$ joining $\xi(0) = K_{t+1}$ to $\xi(1) = K^*$, where $\widetilde{\Exp}$ denotes the exponential map on $\substableK$.
This is certainly possible (by shrinking $\nghbd$ if necessary) because geodesics are locally-minimizing \cite{lee2018introduction}.
Similar to the function $\phi$, define $\psi\colon [0,1] \mapsto \bR$ with 
\[\psi(\eta) \coloneqq \tensor{\grad{h}_{\xi(\eta)}}{E(\eta)}{Y_{\xi(\eta)}},\]
for some parallel vector $E(\eta)$ along $\xi$.
Then, similarly
\begin{gather*}
    \psi'(\eta) = \tensor{D_{\eta}\grad{h}_{\xi(\eta)}}{E(\eta)}{Y_{\xi(\eta)}}
    = \tensor{\hess{h}_{\xi(\eta)}[\xi'(\eta)]}{E(\eta))}{Y_{\xi(\eta)}}.
\end{gather*}
The velocity of any geodesic is a parallel vector field along itself, so by choosing $E(\eta) = \xi'(\eta)$ and using the fundamental lemma of calculus for $\psi$ we obtain that
\begin{gather*}
    \psi(1) = \tensor{\grad{h}_{K_{t+1}}}{F_{t+1}}{Y_{K_{t+1}}} \hspace{-0.0cm}+
     \textstyle \int_0^1 \tensor{\hess{h}_{\xi(\tau)}[\xi'(\tau)]}{\xi'(\tau))}{Y_{\xi(\tau)}} d\tau
\end{gather*}
Note that $\psi(1) =0$ and $\rnorm{\xi'(\tau)}{\xi(\tau)} = \rnorm{F_{t+1}}{K_{t+1}}$ for all $\tau$ as $\xi$ is a geodesic. Thus, by using \cref{eq:mMbounds}, we conclude that
\begin{gather}\label{eq:grad-F-bound}
\small
\begin{aligned}
    m \rnorm{F_{t+1}}{K_{t+1}} \leq \rnorm{\grad{h}_{K_{t+1}}}{K_{t+1}} \leq M \rnorm{F_{t+1}}{K_{t+1}}
\end{aligned}
\end{gather}%
where the last inequality follows by enlarging $M$ when
necessary.
Finally, combining \cref{eq:grad-iter-bound} and \cref{eq:grad-F-bound} at two iterations $t+1$ and $t$, and noticing $\dist(K_{t+1},K^*) = \rnorm{F_{t+1}}{K_{t+1}}$ imply that
\begin{multline}\label{eq:distance-bound}
    \dist(K_{t+1},K^*) \leq {\Large (\sfrac{|1-\eta_t| M^2}{m^2})}~\dist(K_t,K^*) \\
    +  {\Large(\sfrac{\eta_t LM^2}{2m^3})}~\dist(K_t,K^*)^2,
\end{multline}
where $\dist(\cdot,\cdot)$ denotes the Riemannian distance function between two points.
Next, note that the mapping $K \mapsto \mathcal{Q}_K$ is chosen to be smooth such that $\mathcal{Q}_K \succ 0$, therefore as a result of \Cref{lem:dlyap}, the mapping $K \mapsto \lyap(A_{\mathrm{cl}}^\intercal, \mathcal{Q}_K)$ is smooth by composition. By smoothness (in particular continuity) of this mapping and the continuity of the maximum eigenvalue (utilized in the definition of stability certificate $\stabcer_K$ in \Cref{lem:stability-cert}), we can shrink $\nghbd$--if necessary--to obtain a positive constant $c>0$ such that
\begin{equation}\label{eq:stab-cer-lowerbound}
    \stabcer_{K_t} \geq \Large\sfrac{c}{\rnorm{G_t}{K_t}}\geq \sfrac{c\, m}{\left(M \dist(K_t,K^*) \right)},
\end{equation}
where the last inequality follows by combining \cref{eq:bound-nd}, \cref{eq:grad-F-bound} and the fact that $\dist(K_t,K^*) = \rnorm{F_{t}}{K_t}$.
Now, pick $r \in (0,1)$; if we set $\nghbd^* \subset \nghbd \subset \substableK$ such that for any $K_0 \in \nghbd^*$ we have
% \[\dist(K_0,K^*) < \min\{\frac{c m M}{M^2-m^2/4}, \frac{m^3}{2LM^2}\},\]
\[\dist(K_0,K^*) < \min\{\frac{c (\sfrac{M}{m})}{(\sfrac{M}{m})^2 - r/2}, \frac{ r}{(\sfrac{L}{m}) (\sfrac{M}{m})^2}\},\]
then by the choice of stepsize $\eta_t = \min\{\stabcer_{K_t}, 1\}$ and the lower-bound in \cref{eq:stab-cer-lowerbound}, we can claim that
\[{\large\sfrac{|1-\eta_0|M^2}{m^2}}
    +  ({\large\sfrac{\eta_0 LM^2}{2m^3}}) ~ \dist(K_0,K^*) < r.\]
But then, \cref{eq:distance-bound} implies that
\(\dist(K_1,K^*) \leq r \; \dist(K_0,K^*).\)
Therefore, $K_1 \in \nghbd^*$ as $r<1$, and thus by induction we conclude a linear convergence rate to $K^*$. Consequently, \cref{eq:stab-cer-lowerbound} implies that $\stabcer_{K_t} \geq 1$ for large enough $t$, and thus by the choice of step-size, \cref{eq:distance-bound} simplifies to
\begin{equation*}
    \dist(K_{t+1},K^*) \leq  ({\large\sfrac{LM^2}{2m^3}})~ \dist(K_t,K^*)^2,
\end{equation*}
guaranteeing a quadratic convergence rate.
Finally, \Cref{lem:nondegenerate} implies that a critical point is nondegenerate with respect to the induced Riemannian connection on $T\substableK$ if and only if it is so with respect to the Euclidean one. The proof for \ac{ouralgo} with $\euchess$ then follows similarly by redefining $\phi$ and $\psi$ using the Euclidean metric under the usual identification of the tangent bundle.
\end{proof}

%%%%%%%%%%%%%%%%%%%%%%%%%%%%%%%%%%%%
\iftoggle{arXiv-version}{
\begin{proof}[Proof of \Cref{prop:lqrcost}]
By definition, $f\colon \stableK \mapsto \bR$ can be viewed as 
the composition:
\begin{equation}\label{eq:fcompos}
\begin{small}
    \begin{aligned}
        f: K \xrightarrow{\Phi} (A_{\mathrm{cl}}^\intercal, K^\intercal R K +Q) \xrightarrow{\lyap} P_K \xrightarrow{\Psi} \frac{1}{2} \tr{P_K \Sigma_1}.
    \end{aligned}
\end{small}
 \end{equation}
Since the first and last maps are smooth (\ie, linear or quadratic in $K$), we conclude that $f \in C^\infty(\stableK)$ by composition and \Cref{lem:dlyap}.
For any $K \in \stableK$, we can compute its differential at $K$, denoted by $\diff f_K$, using the chain rule:
\[\diff f_K(E) =\diff \Psi_{P_K} \circ \diff \lyap_{(A_{\mathrm{cl}}^\intercal, K^\intercal R K +Q)} \circ \diff \Phi_K (E),\]
for any $E \in T_K\stableK$.  But $\Psi$ is a linear map, and under the usual identification of the tangent bundle we obtain
\[\diff \Phi_K(E) = (E^\intercal B^\intercal, E^\intercal R K + K^\intercal R E).\]
Therefore, by \Cref{lem:dlyap} we claim the followings
\begin{align}\label{eq:dP_K}
    \diff(\lyap \circ \Phi)_K(E) =&
    \lyap\big(A_{\mathrm{cl}}^\intercal, \; E^\intercal (B^\intercal P_K A_{\mathrm{cl}} + R K) \nonumber\\
    &+ (K^\intercal R + A_{\mathrm{cl}}^\intercal P_K B) E\big),\\
    \implies \diff f_K(E) =& \Psi \circ \lyap\big(A_{\mathrm{cl}}^\intercal, E^\intercal (B^\intercal P_K A_{\mathrm{cl}} + R K) \nonumber\\
    &+ (K^\intercal R + A_{\mathrm{cl}}^\intercal P_K B) E\big)\nonumber.
\end{align}
Thus, $\diff f_K(E) = \tensor{E}{R K + B^\intercal P_K A_{\mathrm{cl}}}{Y_K}$ with $Y_K = \lyap(A_{\mathrm{cl}},\Sigma_1)$--by \lyaptrace property--which is well-defined and unique as $A_{\mathrm{cl}}$ is a stability matrix.
As $\diff f_K(E) = E f$, the expression for $\grad{f} \in \fX(\stableK)$ then follows by its definition.
Next, as the Hessian operator is self-adjoint (see \Cref{appendix}\!), in order to obtain $\hess{f}$ we can compute \cref{eq:hess-compute} for any $U, W \in \fX(\stableK)$.
As $\hess{f}[U]|_K$ only depends on the value of $U$ at $K$, it suffices to obtain $\hess{f}_K[U_K]$ at each $K \in \stableK$ with $U_K = E$ for arbitrary $E \in T_K \stableK$. 
To do so, we compute $\tensor{\hess{f}_K[E]}{F}{Y}$ for an arbitrary vector $F \in T_K \stableK$ by extending $F$ to the vector field $W$ along the curve $\gamma\colon t \mapsto K+t E$ with constant coordinates with respect to the global coordinate frame $(\partial_{(i,j)})$. As $\tensor{\hess{f}[U]}{W}{Y}|_K$ only depends on the value of $W_K=F$ and $U_K =E$, how  these vector fields have been extended is arbitrary. By properties of the Riemannian connection and the fact that $W$ can be extended with constant coordinates, here we can compute $\nabla_U W |_K$ in the global coordinate frame $(\partial_{i,j})$ and obtain,
\begin{equation}\label{eq:con_const_vec}
    \nabla_U W |_K = [E]^{k,\ell} \; [F]^{p,q} \; \Gamma^{i,j}_{(k,\ell)(p,q)}(K) \;\partial_{i,j},
\end{equation}
where $\Gamma^{i,j}_{(k,\ell)(p,q)}(K)$ denotes the Christoffel symbols associated with the Riemannian metric $\metric$ at the point $K \in \stableK$.
Therefore, from \cref{eq:hess-compute} we have that
\begin{equation}\label{eq:hess-proof}
    \tensor{\hess{f}[U]}{W}{Y}|_K = E r - \tensor{\grad{f}_K}{\nabla_U W |_K}{Y_K},
\end{equation}
where $r \coloneqq \tensor{\grad{f}}{W}{Y} \in C^\infty(\stableK)$. By the expression obtained for $\grad{f}$ and that $W$ has constant coordinates, the mapping $K \mapsto r(K)$ can be decomposed as:
\begin{multline*}
    K \xrightarrow{\mathrm{Id}\times\mathrm{Id}} (K,K) \xrightarrow{\mathrm{Id}\times (\lyap \circ \Phi)} (K,P_K) \xrightarrow{\Xi}\\
    (A_{\mathrm{cl}}^\intercal, (\grad{f}_K)^\intercal F + F^\intercal \grad{f}_K ) \xrightarrow{\Psi \circ \lyap} r(K),
\end{multline*}
where we used the \lyaptrace property and invariance of trace under transpose to justify the last mapping. Also note that $\Phi$ and $\Psi$ are defined in \cref{eq:fcompos} and $\Xi\colon\Kmatrices \times \Amatrices \mapsto \Amatrices \times \Amatrices$ is defined as above. Therefore, under the usual identification of tangent bundle, for any $(E,G) \in T_{(K,P_K)} (\Kmatrices \times \Amatrices)$ we can compute,
\begin{gather}
    \diff \Xi_{(K,P_K)}[E,G] = \Big(E^\intercal B^\intercal, \; E^\intercal (R + B^\intercal P_K B)F
    + A_{\mathrm{cl}}^\intercal G B F \nonumber\\
    + F^\intercal (R + B^\intercal P_K B) E + F^\intercal B^\intercal G^\intercal A_{\mathrm{cl}} \Big). \label{eq:dXi}
\end{gather}
Therefore, by the chain rule, for any $E \in T_K \stableK$ we have
\begin{gather*}
    \diff r_K(E)= \Psi \circ \diff \lyap \circ \diff \Xi \circ \left(E, \diff(\lyap \circ \Phi)_K (E)\right),
\end{gather*}
where the base points of differentials are understood and dropped for brevity. But then, by \cref{eq:dP_K} and \cref{eq:dXi}, we obtain,
\begin{gather*}
    \diff \Xi \left[E, \diff(\lyap \circ \Phi)_K (E)\right]= \Big( E^\intercal B^\intercal, E^\intercal(R+B^\intercal P_K B)F\\
    + A_{\mathrm{cl}}^\intercal (S_K[E]) B F + F^\intercal (R+B^\intercal P_K B) E + F^\intercal B^\intercal (S_K[E]) A_{\mathrm{cl}} \Big),
\end{gather*}
where $S_K[E]$ is defined in the premise.
Therefore,
\begin{multline*}
    \diff r_K(E) = \Psi \circ \lyap\Big( A_{\mathrm{cl}}^\intercal, \; E^\intercal(R+B^\intercal P_K B)F \\
    + A_{\mathrm{cl}}^\intercal (S_K[E]) B F + A_{\mathrm{cl}}^\intercal (S_K[F]) B E +  E^\intercal B^\intercal (S_K[F]) A_{\mathrm{cl}}\Big),
\end{multline*}
that using the \lyaptrace property can be simplified as, 
\begin{gather*}
\begin{aligned}
    2 \diff r_K(E) =& \tr{E^\intercal B^\intercal (S_K[F]) A_{\mathrm{cl}} Y_K + A_{\mathrm{cl}}^\intercal (S_K[F]) B E Y_K} \\
    &+  \tr{(E^\intercal(R+B^\intercal P_K B) + A_{\mathrm{cl}}^\intercal (S_K[E]) B ) F Y_K}\\
    &+  \tr{F^\intercal((R+B^\intercal P_K B)E + B^\intercal (S_K[E]) A_{\mathrm{cl}} ) Y_K},
\end{aligned}
\end{gather*}
where $Y_K = \lyap(A_{\mathrm{cl}},\Sigma_1)$. Noting that $Y_K$, $P_K$, $S_K[E]$ and $S_K[F]$ are all symmetric, using the cyclic permutation property of trace, we now obtain,
\begin{multline}\label{eq:drE}
    \diff r_K(E) = \tensor{(R+B^\intercal P_K B)E + B^\intercal (S_K[E]) A_{\mathrm{cl}} )}{F}{Y_K}\\
    +\tensor{ B^\intercal (S_K[F]) A_{\mathrm{cl}}}{E}{Y_K}.
\end{multline}
Then, the expression for $\hess{f}$ follows by substituting \cref{eq:drE} and \cref{eq:con_const_vec} in \cref{eq:hess-proof}. Finally, the expression of $\euchess{f}$ can be obtained similarly by threading through the definitions.
\end{proof}

% %%%%%%%%%%%%%%%%%%%%%%%%%%%%%%%%%%%%
\begin{proof}[Proof of \Cref{cor:lqrcost-constrained}]
Smoothness of $h$ and the expression of its gradient follows immediately by  \Cref{prop:extrinsic} and \Cref{prop:lqrcost}.
In order to compute $\hess{h}_K$, we can combine its extrinsic representation as obtained in \Cref{prop:extrinsic} with \cref{eq:hess-proof}, and use the definition of Weingarten map to obtain, \begin{gather*}
\begin{aligned}
    &\tensor{\hess{h}_K[E]}{F}{Y_K}\\
    &=
    \tensor{\tproj (\hess{f}[U]\big|_{\footnotesize\substableK})}{W}{Y}\big|_K + \tensor{\wein_{\nproj (\grad{f}|_{\footnotesize\substableK})}(U)}{W}{Y}\big|_K\\
    &= E r - \tensor{\grad{f}_K}{\nabla_U W |_K}{Y_K}+ \tensor{\nproj \grad{f}_K}{\nproj \nabla_U W|_K}{Y_K}\\
    &= E r - \tensor{\tproj \grad{f}_K}{\nabla_U W |_K}{Y_K},
\end{aligned}
\end{gather*}%
for any $E,F \in T_K\substableK \subset \stableK$, which are extended to vector fields on a neighborhood in $\stableK$ with constant coordinates with respect to the global coordinate frame. The claimed expression of $\hess{h}_K$ then follows by substituting \cref{eq:con_const_vec} and \cref{eq:drE} into the last expression.
\end{proof}
}{}

% %%%%%%%%%%%%%%%%%%%%%%%%%%%%
\begin{proof}[Proof of \ac{slqr} Manifold in \S\ref{subsec:slqr}]
Let the tuple $(x^{(i,j)}|_{(i,j) \in [m]\times[n]})$ denote the component functions of the global smooth chart $(\Kmatrices,\mathrm{vec})$, and define $\Phi\colon\Kmatrices \mapsto \bR^{mn-|D|}$ with
\(\textstyle \Phi(K) = \sum_{(i,j) \notin D} [K]_{i,j} x^{(i,j)}.\)
Then, it is easy to see that $\Phi$ is a smooth submersion, and so is $\Phi|_{\stableK}$ as $\stableK$ is an open submanifold of $\Kmatrices$. Therefore, as $\substableK = \stableK \cap \constraint_D = (\Phi|_{\stableK})^{-1}(0)$, by Submersion Level Set Theorem, we conclude that $\substableK$ is a properly embedded submanifold of dimension $|D|$.  

Furthermore, at any point $K \in \substableK$ and for any tangent vector $E \in T_K \stableK$, we can compute the tangential projection $\tproj\colon T_K\stableK \mapsto T_K\substableK$ as,
\vspace{-0.2cm}
\begin{equation}\label{eq:proj-opt}
     \textstyle \tproj E = \argmin_{\widetilde{E}\in T_K \substableK} \tensor{E-\widetilde{E}}{E-\widetilde{E}}{Y_K}.
\end{equation}
Since $\constraint_D$ is a linear subspace of $\Kmatrices$, we can identify $T_K\substableK$ with $\constraint_D$ itself (a dimension argument). Then, it follows that the unique solution $\widetilde{E}^*$ to the minimization problem above (with linear constraint and strongly convex cost function, as $Y_K\succ 0$ in \Cref{prop:Riemannian-metric}), satisfies $E- \widetilde{E}^* \perp \constraint_D$ with respect to the Riemannian metric at $K$; or equivalently \cref{eq:tproj-linear}.
\end{proof}

% %%%%%%%%%%%%%%%%%%%%%%%%%%%%%%%%%%%%%%%%%
\begin{proof}[Proof of  \ac{olqr} Manifold in \S\ref{subsec:olqr}]
Note that $\constraint_C$ is a linear subspace of $\Kmatrices$ whose dimension depends on the rank of $C$. %
Now, define $\Psi\colon\Kmatrices \mapsto \Kmatrices$ as, 
\(\Psi(K) = K(I_n - C^\dagger C),\)
where $\dagger$ denotes the Moore-Penrose inverse. Note that $\Psi$ is a linear map that is surjective onto its range, denoted by $\mathcal{R}$, which is an $m(n-d)$ dimensional linear subspace of $\Kmatrices$. Therefore, $\Phi\colon\stableK \mapsto \mathcal{R}$ defined as the restriction of $\Psi$ both in domain and codomain, is a smooth submersion. Finally, as $C C^\dagger C = C$, we can observe that $\mathrm{Ker}(\Psi) = \constraint_C$.
Therefore, $\substableK = \constraint_C \cap \stableK = \Phi^{-1}(0)$ and thus, by Submersion Level Set Theorem, $\substableK$ is a properly embedded submanifold of $\stableK$ with dimension $md$. We also conclude that at each $K \in \substableK$, we can canonically identify the tangent space at $K$ as,
\(T_K\substableK = \mathrm{Ker}(\diff \Phi_K) \cong \constraint_C.\)

Next, at any $K \in \substableK$ and for any $E \in T_K \stableK$, the tangential projection of $E$, denoted by $\widetilde{E}^* = \tproj E$, is the unique solution of a minimization problem similar to \cref{eq:proj-opt}. Moreover, under the above identification, $E- \widetilde{E}^* \perp \constraint_C$ (with respect to the Riemannian metric) must be satisfied, or equivalently,
\(\tr{C^\intercal L^\intercal (E-\widetilde{E}^*)Y_K} = 0, \quad \forall L \in \Lmatrices.\)
Here, $Y_K = \lyap(A_{\mathrm{cl}},\Sigma_1)$ is positive definite and since $C$ is assumed to be full-rank, $CY_K C^\top$ is positive definite.
Hence, we conclude that
$\tproj E = L^* C$ with $L^* \in \Lmatrices$ being the unique solution of \cref{eq:tangent_olqr}.

Finally, at each point $K \in \substableK$, we denote the global coordinate functions of $\Lmatrices$ (with slight abuse of notation) by the tuple $(x^{i,j})$ for $(i,j) \in D:= [m]\times[d]$ and its corresponding global coordinate frame by $(\partial_{(i,j)})$. Recall that $C$ has full-rank and consider the identification of $T_K\substableK \cong \constraint_C$ described above. Then, 
the (constant) global vector fields $(\widetilde{\partial}_{(i,j)} = \partial_{(i,j)} C)$ with $(i,j) \in D$,
form a global smooth frame for $T\substableK$ as they are linearly independent on $\substableK$.
Therefore, the coordinates of the covariant Hessian $h_{;(k,\ell)(p,q)}(K)$ with respect to this frame can be computed by substituting $E = \partial_{(k,\ell)} C$ and $F = \partial_{(p,q)} C$ in \Cref{cor:lqrcost-constrained} for each $(k,\ell),(p,q) \in D$--similar to the \ac{slqr} case. It is worth noting that the sparsity pattern in $E$, $F$ and Christoffel symbols can simplify the computation; we will not delve further into this issue due to space limitations.
\end{proof}

\balance
\bibliographystyle{ieeetr}
\bibliography{citations}
\begin{IEEEbiography}
[{\includegraphics[width=1in,height=1.3in,clip,keepaspectratio]{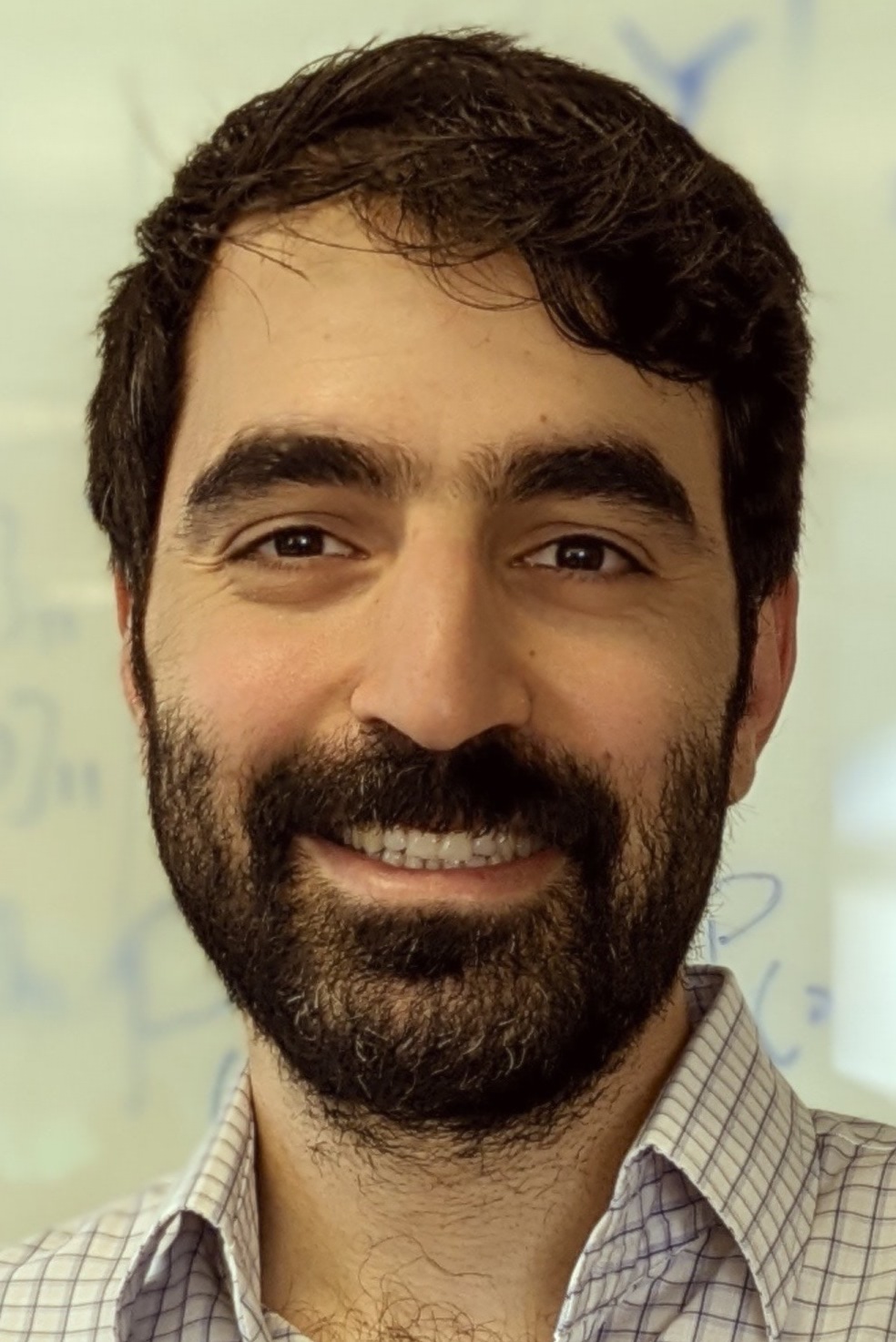}}]
{Shahriar Talebi} (Student Member, IEEE) 
received the Ph.D. degree in aeronautics and astronautics, specializing in control theory, and the M.Sc. degree in Mathematics, focusing on differential geometry, both from the University of Washington, Seattle, WA, USA, in 2023. He also received the B.Sc. degree from Sharif University of Technology, Tehran, Iran, in 2014, and the M.Sc. degree from the University of Central Florida, Orlando, FL, in 2017, both in electrical engineering.

His research interests include control theory, differential geometry, learning for control, networked dynamical systems, and game theory.

Dr. Talebi was the recipient of the 2022 Excellence in Teaching Award at UW. He is also a recipient of a number of scholarships, including the William E. Boeing Endowed Fellowship, Paul A. Carlstedt Endowment, and Latvian Arctic Pilot-A. Vagners Memorial Scholarship (UW, 2018–2019), as well as the Frank Hubbard Engineering Scholarship (UCF, 2017).

% received his Ph.D. in aeronautics and astronautics, specializing in control theory, and his M.Sc. in Mathematics, focusing on differential geometry, from the University of Washington, Seattle, WA, USA, in 2023. He holds a B.Sc. from Sharif University of Technology, Tehran, Iran and an M.Sc. from the University of Central Florida, Orlando, FL, both in electrical engineering.
    
    % Mr. Talebi received the 2022 Excellence in Teaching Award at UW. He is also a recipient of a number of scholarships, including the William E. Boeing Endowed Fellowship, Paul A. Carlstedt Endowment, and Latvian Arctic Pilot–A. Vagners Memorial Scholarship (UW, 2018-2019), as well as the Frank Hubbard Engineering Scholarship (UCF, 2017).

    % His research interests include control theory, differential geometry, learning for control, networked dynamical systems, and game theory.
\end{IEEEbiography}
\vfill

\begin{IEEEbiography}
[{\includegraphics[width=1in,height=1.25in,clip,keepaspectratio]{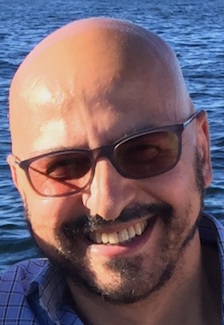}}]
    {Mehran Mesbahi} (Fellow, IEEE) received his Ph.D. degree in electrical engineering-systems from University of Southern California, Los Angeles, CA, USA, in 1996.
    
    M. Mesbahi was a member of the Guidance, Navigation, and Analysis group at JPL, Pasadena, CA, USA, from 1996–2000 and an Assistant Professor of Aerospace Engineering and Mechanics at the University of Minnesota from 2000–2002.
    He is currently a Professor of Aeronautics and Astronautics, Adjunct Professor of Electrical and Computer Engineering and Mathematics at the University of Washington, and the Executive Director of the Joint Center for Aerospace Technology Innovation. He was the recipient of NSF CAREER Award, NASA Space Act Award, UW Distinguished Teaching Award, UW College of Engineering Innovator Award for Teaching, and is a member of Washington State Academy of Sciences.
    
    His research interests include distributed and networked aerospace systems, autonomy, and system and control theory.
\end{IEEEbiography}

\end{document}